\documentclass[12pt, reqno]{amsart}
\usepackage{amssymb}
\usepackage[abbrev]{amsrefs}
\usepackage{amsmath}
\usepackage{color}
\usepackage{bm}
\usepackage{enumitem}

\newcommand{\Id}{\ensuremath{{\mathrm{Id}}}}
\newcommand{\ww}{\ensuremath{{\bm{w}}}}
\newcommand{\uu}{\ensuremath{{\bm{u}}}}
\newcommand{\RB}{\ensuremath{\mathcal{R}}}
\newcommand{\OB}{\ensuremath{\mathcal{O}}}
\newcommand{\BB}{\ensuremath{\mathcal{B}}}
\newcommand{\XB}{\ensuremath{\mathcal{X}}}
\newcommand{\YB}{\ensuremath{\mathcal{Y}}}
\newcommand{\UB}{\ensuremath{\mathcal{U}}}
\newcommand{\VB}{\ensuremath{\mathcal{V}}}
\newcommand{\TB}{\ensuremath{\mathcal{T}}}
\newcommand{\HH}{\ensuremath{\bm{H}}}
\newcommand{\Ind}{\ensuremath{\mathbf{1}}}
\newcommand{\udf}{\ensuremath{\bm{\varphi}}}
\newcommand{\vv}{\ensuremath{\bm{v}}}
\newcommand{\xx}{\ensuremath{\bm{x}}}
\newcommand{\yy}{\ensuremath{\bm{y}}}
\newcommand{\zz}{\ensuremath{\bm{z}}}
\newcommand{\EB}{\ensuremath{\mathcal{E}}}

\newcommand{\ee}{\ensuremath{\bm{e}}}
\newcommand{\Sym}{\ensuremath{\mathbb{S}}}
\newcommand{\EE}{\ensuremath{\mathbb{E}}}
\newcommand{\FF}{\ensuremath{\mathbb{F}}}
\newcommand{\ZZ}{\ensuremath{\mathbb{Z}}}
\newcommand{\XX}{\ensuremath{\mathbb{X}}}
\newcommand{\YY}{\ensuremath{\mathbb{Y}}}
\newcommand{\TT}{\ensuremath{\mathbb{T}}}
\newcommand{\UU}{\ensuremath{\mathbb{U}}}
\newcommand{\unc}{\ensuremath{\bm{k}}}

\newcommand{\aunc}{\ensuremath{\bm{\tilde{k}}}}
\newcommand{\dom}{\ensuremath{\bm{\delta}}}
\newcommand{\adom}{\ensuremath{\bm{\tilde{\delta}}}}
\newcommand{\supp}{\mathop{\mathrm{supp}}\nolimits}

\newcommand{\sgn}{\mathop{\mathrm{sign}}\nolimits}
\newcommand{\mm}{\ensuremath{{\bm{m}}}}
\newcommand{\NN}{\ensuremath{\mathbb{N}}}
\newcommand{\CC}{\ensuremath{\mathbb{C}}}
\newcommand{\RR}{\ensuremath{\mathbb{R}}}

\let\abs=\envert

\let\norm=\enVert

\AtBeginDocument{
\def\MR#1{}
}

\newtheorem{theorem}{Theorem}[section]
\newtheorem{proposition}[theorem]{Proposition}
\newtheorem{corollary}[theorem]{Corollary}
\newtheorem{lemma}[theorem]{Lemma}

\newtheorem{question}[theorem]{Question}
\theoremstyle{definition}

\theoremstyle{remark}

\numberwithin{equation}{section}
\allowdisplaybreaks

\begin{document}

\title[Fourier coefficients in power-weighted $L_2$-spaces]{Fourier coefficients of functions in power-weighted $L_2$-spaces and conditionality constants of bases in Banach spaces}

\thanks{The author acknowledges the support of the Spanish Ministry for Science, Innovation, and Universities under Grant PGC2018-095366-B-I00 for \emph{An\'alisis Vectorial, Multilineal y Aproximaci\'on.}}

\author[J. L. Ansorena]{J. L. Ansorena}\address{Department of Mathematics and Computer Sciences\\
Universidad de La Rioja\\
Logro\~no 26004\\
Spain} \email{joseluis.ansorena@unirioja.es}



\begin{abstract}
We prove that, given $2<p<\infty$, the Fourier coefficients of functions in $L_2(\TT, \abs{t}^{1-2/p}\, dt)$ belong to $\ell_p$, and that, given $1<p<2$, the Fourier series of  sequences in $\ell_p$ belong $L_2(\TT, \abs{t}^{2/p-1}\, dt)$. Then, we apply these results to the study of conditional Schauder bases and conditional almost greedy bases in Banach spaces. Specifically, we prove that, for every $1<p<\infty$ and every $0\le \alpha<1$, there is a Schauder basis of $\ell_p$ whose conditionality constants grow as $(m^\alpha)_{m=1}^\infty$,  and there is an almost greedy basis of $\ell_p$ whose conditionality constants grow as $((\log m)^\alpha)_{m=2}^\infty$.
\end{abstract}

\maketitle

\section{Introduction}\noindent
To contextualize the study carried out in this paper, we bring up a classical theorem that applies, in particular, to Hilbert spaces.

\begin{theorem}[\cite{GurGur1971}*{Theorems 1 and 2} and \cite{James1972}*{Theorems 2 and 3}]\label{thm:EJ}
Let $\XB=(\xx_n)_{n=1}^\infty$ be a Schauder basis of superreflexive Banach space $\XX$. Suppose that $\XB$ is semi-normalized, i.e.,
\[
\inf_n\norm{\xx_n}>0, \quad \sup_n \norm{\xx_n}<\infty.
\]
Then, there are $1<q\le r<\infty$ such that $\XB$ is $r$-Besselian and $q$-Hilbertian. That is, if $\XB^*=(\xx_n^*)_{n=1}^\infty$ is the dual Schauder basis of $\XB$, the coefficient transform
\[
f\mapsto (\xx_n^*(f))_{n=1}^\infty
\]
is a bounded operator from $\XX$ into $\ell_r$, and the series transform
\[
(a_n)_{n=1}^\infty\mapsto \sum_{n=1}^\infty a_n\, \xx_n
\]
is a bounded operator from $\ell_q$ into $\XX$.
\end{theorem}

Theorem~\ref{thm:EJ} leads naturally to pose the following general problem.

\begin{question}\label{question:BH}
Let $\XB$ be a semi-normalized Schauder basis of a superreflexive Banach space $\XX$. The non-trivial intervals
\begin{align*}
J_B[\XB,\XX]&=\{r\in[1,\infty] \colon \text{$\XB$ is  $r$-Besselian}\}\; \text{and}\\
J_H[\XB,\XX]&=\{q\in[1,\infty] \colon \text{$\XB$ is  $q$-Hilbertian}\},
\end{align*}
called the Besselian and Hilbertian intervals of $\XB$ in $\XX$, respectively, contain valuable information on the geometry of the basis $\XB$. So they are worth studying.
\end{question}

In this paper, we address Question~\ref{question:BH} for the trigonometric system in Hilbert spaces arising from power weights. Given $\lambda\in\RR$, we consider the weight
\[
\ww_\lambda\colon[-1/2,1/2]\to[0,\infty], \quad t\mapsto \abs{t}^{\lambda}.
\]
If $\lambda>-1$ and $n\in\ZZ$, then the trigonometric function
\[
\tau_n\colon[-1/2,1/2]\to\CC, \quad \tau_n(t)= e^{2\pi i n t },
\]
belongs to the complex Hilbert space
\[
\HH_\lambda=L_2(\ww_\lambda,\CC).
\]
Moreover, the norm of $\tau_n$ does not depend on $n$. The natural arrangement of the trigonometric system $(\tau_n)_{n=-\infty}^\infty$ will be denoted by $\TB$, that is, $\TB=(\phi_n)_{n=0}^\infty$, where $\phi_{2n}=\tau_{-n}$ and  $\phi_{2n+1}=\tau_{n+1}$ for all $n\in\NN\cup\{0\}$. As we are also interested in Hilbert spaces over $\RR$,  we consider the real-valued counterpart of $\TB$. We define $\TB^{\,\RR}=(\phi_n^\RR)_{n=0}^\infty$ by $\phi_0^\RR=1$ and
\[
\phi^\RR_{2k-1}(t)=\cos(2\pi k t), \quad \phi^\RR_{2k}(t)=\sin(2\pi k t), \quad t\in[-1/2,1/2], \; k\in\NN.
\]
Let us record the obvious relations between $\TB$ and $\TB^{\,\RR}$. We have
\begin{align}
\phi^\RR_{2k-1}&=\frac{\phi_{2k-1}+\phi_{2k}}{2},\quad \phi^\RR_{2k}=\frac{\phi_{2k-1}-\phi_{2k}}{2i}, \quad k\in\NN, \;\mbox{ and}
\label{eq:cossin}\\
\phi_{2k-1}&=\phi^\RR_{2k-1}+i\phi^\RR_{2k}, \quad \phi_{2k}=\phi^\RR_{2k-1}-i\phi^\RR_{2k}, \quad k\in\NN.
\label{eq:exp}
\end{align}

In the case when $-1<\lambda<1$, the conjugate-function operator
\[
f=\sum_{n=-\infty}^{\infty} a_n \tau_n \mapsto \widetilde{f}= \sum_{n=-\infty}^{\infty} \sgn(n) a_n \tau_n
\]
is bounded on $\HH_\lambda$  \cite{Babenko1948} and, hence, $\TB$ is a Schauder basis of  $\HH_\lambda$. Of course, this result can be derived from the fact that $\ww_\lambda$ is a Muckenhoupt $A_2$ weight (see \cite{HMW1973}*{Theorem 8}). Taking into consideration  \eqref{eq:cossin}, we infer from \cite{AAW2019}*{Lemma 2.4} that $\TB^{\,\RR}$ is a semi-normalized basis of  both $\HH_\lambda$ and  its real-valued counterpart
\[
\HH^\RR_\lambda=L_2(\ww_\lambda,\RR).
\]

Here, we contribute to the understanding of the trigonometric system by computing the Besselian and Hilbertian intervals of $\TB$ and $\TB^{\,\RR}$ regarded as systems in $\HH_\lambda$ and $\HH^\RR_\lambda$, respectively. Namely, we will prove the following result.

\begin{theorem}\label{thm:Main}
Let $0\le \alpha<1$. Define $1<q_\alpha\le 2 \le r_\alpha<\infty$ by
\[
q_\alpha=\frac{2}{1+\alpha}, \quad r_\alpha=\frac{2}{1-\alpha}.
\]
Then,
\begin{enumerate}[label=(\roman*), leftmargin=*, widest=iii]
\item $J_B[\TB,\HH_{-\alpha}]=J_B[\TB^{\,\RR},\HH_{-\alpha}]=J_B[\TB^{\,\RR},\HH^\RR_{-\alpha}]=[2,\infty]$,
\item $J_H[\TB,\HH_{-\alpha}]=J_H[\TB^{\,\RR},\HH_{-\alpha}]=J_H[\TB^{\,\RR},\HH^\RR_{-\alpha}]=[1,q_\alpha]$,
\item $J_B[\TB,\HH_\alpha]=J_B[\TB^{\,\RR},\HH_\alpha]=J_B[\TB^{\,\RR},\HH^\RR_\alpha]=[r_\alpha,\infty]$, and
\item $J_H[\TB, \HH_\alpha]=J_H[\TB^{\,\RR}, \HH_\alpha]=J_H[\TB^{\,\RR}, \HH^\RR_\alpha]=[1,2]$.
\end{enumerate}
\end{theorem}
Notice that, since the trigonometric system is an orthogonal basis of $L_2([-1/2,1/2])$, the case $\alpha=0$ in Theorem~\ref{thm:Main} is just a consequence of combining Bessel's inequality with Riesz-Fischer Theorem.

We complement our research by applying Theorem~\ref{thm:Main} to the study of conditional bases in $\ell_p$-spaces. Since every Banach space with a Schauder basis has a conditional Schauder basis by a classical theorem of Pełczy\'{n}ski and Singer \cite{PelSin1964}, to obtain information on the structure of a given space by means of its conditional bases we must study certain additional features of the bases. In this regard, focusing on the conditionality parameters of the bases is an inviting line of research. Let us abridge the necessary terminology and background on this topic.

Given a Schauder basis $\XB=(\xx_n)_{n=1}^\infty$ of a Banach space $\XX$ with dual basis $(\xx_n^*)_{n=1}^\infty$, we put
\[
\unc_m=\unc_m[\XB,\XX]=\sup_{\abs{A}\le m} \norm{S_A}, \quad m\in\NN,
\]
where $S_A=S_A[\XB,\XX]\colon\XX\to\XX$ is the  the coordinate projection on the finite subset $A\subseteq\NN$, i.e.,
\[
S_A(f) = \sum_{n\in A} \xx_n^*(f) \, \xx_n, \quad f\in\XX.
\]
Since $\XB$ is unconditional if and only if $\sup_m \unc_m<\infty$, the growth of the sequence $(\unc_m)_{m=1}^\infty$
measures how far $\XB$ is from being unconditional.

An application of Theorem~\ref{thm:EJ} gives that if $\XX$ is superreflexive, then there is $0\le\alpha<1$ such that
\begin{equation}\label{eq:UCSR}
\unc_m[\XB,\XX]\lesssim m^{\alpha}, \quad m\in\NN.
\end{equation}
(see Theorem~\ref{lem:BHCC} and Equation~\eqref{eq:Comlplq} below). In fact, this property characterizes superreflexivity (see \cite{AAW2019}*{Corollary 3.6}). The authors of \cite{GW2014} proved that the estimate \eqref{eq:UCSR} is optimal for Hilbert spaces and, more generally, $\ell_p$ spaces for $1<p<\infty$. To be precise, Garrig\'os and Wojtaszczyk proved that for every $0\le \alpha<1$ there is a Schauder basis $\XB_\alpha$ of $\ell_p$ with
\[
\unc_m[\XB_\alpha,\ell_p]\gtrsim  m^{\alpha},  \quad m\in\NN.
\]
However, they did not compute the exact growth of the sequence $(\unc_m)_{m=1}^\infty$. We will take advantage of Theorem~\ref{thm:Main} to complement their study. Namely, we will prove the following.
\begin{theorem}\label{thm:CPlp}
For every $0\le \alpha <1$ and $1<p<\infty$ there is a Schuader basis $\XB_\alpha$ of $\ell_p$ with
\[
\unc_m[\XB_\alpha,\ell_p]\approx  m^{\alpha},  \quad m\in\NN.
\]
\end{theorem}

Another way to get  information on the structure of a given space by means of its conditional bases it to restrict the discussion on their existence by imposing certain distinctive properties. These additional properties can be imported to Banach space theory from greedy approximation theory, where we find interesting types of bases, such as almost greedy bases, which are suitable to implement the greedy algorithm and yet they need not be unconditional. In this regard, we point out that superreflexivity can also be characterized in terms of the conditionality parameters of almost greedy bases. In fact, since the GOW-method invented in \cite{GW2014} gives rise to almost greedy bases, we can safely replace `quasi-greedy' with `almost greedy' in  \cite{AAW2019}*{Corollary 3.6}. Thereby, a Banach space $\XX$ is superreflexive if and only if every almost greedy basis $\YB$ of a Banach space $\YY$ finitely representable in $\XX$ satisfies
\begin{equation}\label{eq:SRAG}
\unc_m[\YB,\YY]\lesssim  (\log m)^{\alpha},  \quad m\ge 2,
\end{equation}
for some $\alpha<1$. As well as \eqref{eq:UCSR}, the estimate \eqref{eq:SRAG} is optimal in the sense that for every $0\le \alpha<1$ and every $1<p<\infty$ there is an almost greedy basis $\YB_\alpha$ of $\ell_p$ with
\begin{equation*}
\unc_m[\YB_\alpha,\ell_p]\gtrsim (\log m)^{\alpha},  \quad m\in\NN
\end{equation*}
(see \cite{GW2014}*{Theorem 1.2}). In this paper, we improve this result by computing the growth of the sequence $(\unc_m)_{m=1}^\infty$. To be precise, we will prove the following result.

\begin{theorem}\label{thm:CPAGlp}
Let $1<p<\infty$, and let $\abs{1/2-1/p} \le \alpha <1$. Then, there is an almost greedy basis $\YB_\alpha$ of $\ell_p$ with
\[
\unc_m[\YB_\alpha,\ell_p]\approx  (\log m)^{\alpha},  \quad m\ge 2.
\]
\end{theorem}
Note that Theorem~\ref{thm:CPAGlp} is, in a sense, the almost greedy counterpart of Theorem~\ref{thm:CPlp}.

The article is structured in three more sections. In Section~\ref{sect:preliminary}, we record some general results on Hilbertian an Besselian intervals, as well as some general results on conditionality parameters. Section~\ref{sect:Main} revolves around the proof of Theorem~\ref{thm:Main}. Section~\ref{sect:CPSchauder} is devoted to prove Theorem~\ref{thm:CPlp}. In turn, Section~\ref{sect:CPAG} is geared toward the proof of Theorem~\ref{thm:CPAGlp}.

Throughout this paper, we employ standard notation and terminology commonly used in Fourier Analysis, Functional Analysis and Approximation Theory, as the reader will find, e.g., in the monographs \cites{Duo2001,AlbiacKalton2016,LinTza1979}. Other more specific terminology will be introduced in context when needed.

\section{Preliminary results}\label{sect:preliminary}\noindent
For broader applicability, we will consider Question~\ref{question:BH} within a setting more general than that of Schauder bases. A \emph{biorthogonal system} in a Banach space $\XX$ over the real or complex field $\FF$ is a sequence $\OB=(\xx_n,\xx_n^*)_{n=1}^\infty$ in $\XX\times\XX^*$ with $\xx_n^*(\xx_k)=\delta_{n,k}$ for all $(n,k)\in\NN^2$. The Besselian and Hilbertian intervals of the biorthogonal system $\OB$ are defined analogously to those of a Schauder basis.
The interval $J_B[\OB,\XX]$ is nonempty, that is, $\infty\in J_B[\OB,\XX]$, if and only if $(\xx_n^*)_{n=1}^\infty$ is norm-bounded. Similarly, $1$ belongs to $J_B[\OB,\XX]$, so that  $J_B[\OB,\XX]\not=\emptyset$, if and only if $(\xx_n)_{n=1}^\infty$ is norm-bounded. Notice that  $(\xx_n)_{n=1}^\infty$ and  $(\xx_n^*)_{n=1}^\infty$ are simultaneously norm-bounded if and only if  $(\xx_n)_{n=1}^\infty$ is semi-normalized and
\begin{equation}\label{eq:MB}
\sup_n \norm{\xx_n} \norm{\xx_n^*}<\infty.
\end{equation}
If  \eqref{eq:MB} holds, we say that the biorthogonal system $\OB$ is \emph{$M$-bounded}.

If $\XB=(\xx_n)_{n=1}^\infty$ is a Schauder basis with dual basis $(\xx_n^*)_{n=1}^\infty$ then, since the partial sum projections with respect to $\XB$ are  uniformly bounded, $(\xx_n,\xx_n^*)_{n=1}^\infty$ is an \emph{$M$-bounded} biorthogonal system.

Given a biorthogonal system $\OB=(\xx_n,\xx_n^*)_{n=1}^\infty$, the set
\[
J[\OB,\XX]:=J_H[\OB,\XX]\cap J_B[\OB,\XX]
\]
is either empty or a singleton. Moreover, if $J[\OB,\XX]=\{r\}$, then $(\xx_n)_{n=1}^\infty$ is equivalent to the unit vector system of $\ell_r$. The existence of an unconditional basis for $\XX$ enables us to obtain more significant information.

Given a set $A\subseteq[1,\infty]$, we put
\[
A'=\{q':=q/(q-1) \colon p\in A\},
\]
and we say that $A$ and $A'$ are conjugate sets.

\begin{proposition}\label{lem:latticeEst}
Let $\XX$ be a Banach space with an unconditional basis $\UB=(\uu_j)_{j=1}^\infty$. Suppose that the lattice structure on $\XX$ induced by $\UB$ satisfies a lower $r_0$-estimate and an upper $q_0$-estimate, $1\le q_0\le r_0\le\infty$. Then $J_H[\OB,\XX]\subseteq[1,r_0]$ and $J_B[\OB,\XX]\subseteq[q_0,\infty]$ for any semi-normalized $M$-bounded biorthogonal system $\OB=(\xx_n,\xx_n^*)_{n=1}^\infty$  in $\XX$. In particular, if $\XX=\ell_p$, $1\le p\le\infty$, then  $J_B[\OB,\XX]\subseteq[p,\infty]$ and $J_H[\OB,\XX]\subseteq[1,p]$ (we replace $\ell_\infty$ with $c_0$ if $p=\infty$).
\end{proposition}

\begin{proof}
Let $\UB^*=(\uu_j^*)_{j=1}^\infty$ be the dual basis of $\UB$. Suppose that $\OB$ is $q$-Hilbertian, $1\le q \le\infty$. By the Cantor diagonal technique, passing to a subsequence we can suppose that $(\uu_j^*(\xx_n))_{n=1}^\infty$ converges for every $j\in\NN$. Then, replacing $\xx_n$ with $\xx_{2n}-\xx_{2n-1}$ and  $\xx_n^*$ with $(\xx_{2n}^*-\xx_{2n-1}^*)/2$ we can suppose that $\lim_n \uu_j^*(\xx_n)=0$ for every $j\in\NN$. By the gliding hump technique, passing to a subsequence we can suppose that $(\xx_n)_{n=1}^\infty$ is equivalent to a block basic sequence  with respect to $\UB$, say $\YB=(\yy_n)_{n=1}^\infty$. On the one hand, the mapping
\[
(a_n)_{n=1}^\infty\mapsto \sum_{n=1}^\infty a_n\, \yy_n
\]
defines a bounded operator from $\ell_q$ into $\XX$. On the other hand, since $\YB$ is semi-normalized, there is a constant $C$ such that
\[
\norm{f}_{r_0} \le C \norm{\sum_{n=1}^\infty a_n\, \yy_n}, \quad f=(a_n)_{n=1}^\infty\in c_{00}.
\]
We infer that $q\le r_0$.

The above argument proves the Hilbertian part. In order to prove the Besselian part, we assume that  $q_0>1$;  otherwise there is nothing to prove. Since every semi-normalized block basic sequence $(\zz_n)_{n=1}^\infty$ with respect to $\UB$  satisfies the estimate
\[
\norm{\sum_{n=1}^\infty a_n\, \zz_n}\le C \norm{f}_{q_0}, \quad f=(a_n)_{n=1}^\infty\in c_{00},
\]
for some constant $C$, $\UB$ does not have any block basic sequence equivalent to the unit vector system of $\ell_1$. By \cite{AlbiacKalton2016}*{Theorem 3.3.1}, $\UB^*$ is a basis of $\XX^*$. By \cite{LinTza1979}*{Proposition 1.f.5}, $\UB^*$ induces on $\XX^*$ a lattice structure which satisfies a lower $q_0'$-estimate. Let $\OB^*$ denote the biorthogonal system $(\xx_n^*,h_\XX(\xx_n))_{n=1}^\infty$, where $h_\XX\colon\XX\to \XX^{**}$ is the bidual map. By the already proved Hilbertian part, $J_H[\OB^*,\XX^*]\subseteq[1,q_0']$. By duality, $J_B[\OB,\XX]\subseteq (J_H[\OB^*,\XX^*])'$. Consequently,
$
J_B[\OB,\XX]\subseteq [1,q_0']'=[q_0,\infty].
$

We conclude the proof by noticing that $\ell_p$ a $p$-convex and $p$-concave Banach lattice.
\end{proof}

To establish our results as accurately as possible, we will need notions more general than $p$-Besselian and $p$-Hilbertian bases. A \emph{sequence space} will be a Banach space or, more generally, a quasi-Banach space $\UU\subseteq\FF^\NN$ for which the unit vector system
\[
\EB=(\ee_n)_{n=1}^\infty
\] is a normalized $1$-unconditional basis. We say that a biorthogonal system $\OB$ in a Banach space $\XX$ is $\UU$-Besselian (resp., $\UU$-Hilbertian) if the coefficient transform (resp., the series transform) is a bounded operator from $\XX$ into $\UU$ (resp., from $\UU$ into $\XX$).

Given sequences $\XB=(\xx_n)_{n=1}^\infty$ and $\YB=(\yy_n)_{n=1}^\infty$ in Banach spaces $\XX$ and $\YY$, respectively, their direct sum is the sequence $\XB\oplus\YB=(\uu_n)_{n=1}^\infty$ in $\XX\oplus\YY$ given by
\[
\uu_{2n-1}=(\xx_n,0), \quad \uu_{2n}=(0,\yy_n), \quad n\in\NN.
\]
We put $\XB^2=\XB\oplus\XB$, and we call $\XB^2$ the square of $\XB$. If $\XB$ is a semi-normalized Schauder basis of $\XX$, and $\YB$ is a semi-normalized Schauder basis of $\YY$, then $\XB\oplus\YB$ is a seminormalized Schauder basis of $\XX\oplus\YY$ whose dual basis is, modulus the natural identification of $(\XX\oplus\YY)^*$ with $\XX^*\oplus\YY^*$, $\XB^*\oplus\YB^*$.  Given biorthogonal systems $\OB=(\xx_n,\xx_n^*)_{n=1}^\infty$ and $\RB=(\yy_n,\yy_n^*)_{n=1}^\infty$ in Banach spaces $\XX$ and $\YY$, respectively, their direct sum is the biorthogonal system $\OB\oplus\RB$ in $\XX\oplus\YY$ whose first and second components are  $(\xx_n)_{n=1}^\infty\oplus (\yy_n)_{n=1}^\infty$ and $(\xx^*_n)_{n=1}^\infty\oplus (\yy^*_n)_{n=1}^\infty$, respectively.

The \emph{rotation} of a sequence $\XB=(\xx_n)_{n=1}^\infty$ is a Banach space $\XX$ will be sequence $\XB_\diamond=(\yy_n)_{n=1}^\infty$  given by
\[
\yy_{2n-1}=\frac{\xx_{2n-1}-\xx_{2n}}{\sqrt{2}},\quad
\yy_{2n}=\frac{\xx_{2n-1}+\xx_{2n}}{\sqrt{2}}, \quad n\in\NN.
\]
If $\XB$ is Schauder basis, then $\XB_\diamond$ is a Schauder basis with dual basis $(\XB^*)_\diamond$. Given a biorthogonal system $\OB=(\xx_n,\xx_n^*)_{n=1}^\infty$  in $\XX$, its rotation is the biorthogonal system $\OB_\diamond$ in $\XX$ whose first and second components are the rotations of $(\xx_n)_{n=1}^\infty$ and $(\xx^*_n)_{n=1}^\infty$, respectively.

We say that a sequence space $\UU$ is lattice isomorphic to its square if the unit vector system of $\UU$ is equivalent to its square.

Lemma~\ref{lem:gathered} below, whose straightforward proof we omit, gathers some properties of these notions that we will need.

\begin{lemma}\label{lem:gathered}
Let $\UU$ be a sequence space, and let $\OB=(\xx_n,\xx_n^*)_{n=1}^\infty$ be a biorthogonal system in a Banach space $\XX$. Suppose that $\OB$ is $\UU$-Besselian (resp., $\UU$-Hilbertian).
\begin{enumerate}[label=(\roman*), leftmargin=*, widest=iii]
\item if $(\lambda_n)_{n=1}^\infty$ is a semi-normalized sequence in $\FF$, then the perturbed system $(\lambda_n\, \xx_n,\lambda_n^{-1} \xx_n^*)_{n=1}^\infty$ is $\UU$-Besselian (resp., $\UU$-Hilbertian).

\item Suppose that $\UU$ is lattice isomorphic to its square. Then, the rotated system $\OB_\diamond$ is  $\UU$-Besselian (resp., $\UU$-Hilbertian). Moreover, if $\RB$ is a $\UU$-Besselian (resp., $\UU$-Hilbertian) biorthogonal system of a Banach space $\YY$, then $\OB\oplus\RB$ is a $\UU$-Besselian (resp., $\UU$-Hilbertian) biorthogonal system of $\XX\oplus\YY$.

\item Suppose that $\UU$ and $\XX$ are real-valued spaces. Then, $\OB$ is a $\UU^{\,\CC}$-Besselian (resp., Hilbertian) biorthogonal system in the complexification $\XX^{\,\CC}$ of $\XX$.
\end{enumerate}
\end{lemma}

The authors of  \cite{AAW2019} introduced an alternative quantitative measure of the unconditionality of a biorthogonal system $\OB=(\xx_n,\xx_n^*)_{n=1}^\infty$ which is more accurate in some situations. Set
\[
\NN[m]=\{n\in\ZZ \colon 1\le n\le m\}
\]
and define
\[
\aunc_m=\aunc_m[\OB,\XX] :=\sup\{  \norm{S_A(f)} \colon \norm{f}\le 1,\; \supp(f) \subseteq\NN[m],\; A\subseteq\NN\}.
\]
We have $\aunc_m\le\unc_m$ for all $m\in\NN$, and $\sup_m\aunc_m=\sup_m\unc_m$.

If $\OB$ and $\RB$ are biorthogonal systems in Banach spaces $\XX$ and $\YY$, then
\begin{align}
\aunc_m[\OB\oplus \RB,\XX\oplus\YY]&= \max\{\aunc_{\lceil m/2\rceil}[\OB,\XX], \aunc_{\lfloor m/2\rfloor}[\RB,\YY]\}, \label{eq:AUNCCDS}\\
\unc_m[\OB\oplus \RB,\XX\oplus\YY]&=\max\{\unc_m[\OB,\XX], \unc_m[\RB,\YY]\}\label{eq:UNCCDS}.
\end{align}

Loosely speaking, we could say that $\OB\oplus \RB$ inherits naturally the properties of $\OB$ and $\RB$. In contrast, `rotating' $\OB\oplus \RB$ gives rise to more interesting situations. Set $\OB\diamond\RB:=(\OB\oplus \RB)_\diamond$ and define
\begin{align*}
\adom_m[\OB,\XX,\RB,\YY]&=\sup_{\substack{A\subseteq\NN[m]\\ (a_n)_{n\in A}\in\FF^A\setminus\{0\}}}
\frac{\norm{\sum_{n\in A} a_n \, \xx_n}}{\norm{\sum_{n\in A} a_n \, \yy_n}},\quad m\in\NN,\\
\dom_m[\OB,\XX,\RB,\YY]&=\sup_{\substack{ \abs{A}\le m\\ (a_n)_{n\in A}\in\FF^A\setminus\{0\}}}
\frac{\norm{\sum_{n\in A} a_n \, \xx_n}}{\norm{\sum_{n\in A} a_n\, \yy_n}}, \quad m\in\NN.
\end{align*}
\begin{lemma}\label{lem:ccdom}
Let $\OB$ and $\RB$ be biorthogonal systems in Banach spaces $\XX$ and $\YY$ respectively. Then
\begin{align*}
\unc_{2m}[\OB\diamond \RB,\XX\oplus\YY]&\ge \frac{1}{2}\max\{\dom_{m}[\OB,\XX,\RB,\YY], \dom_{m}[\RB,\YY,\OB,\XX]\},\; \text{and}\\
\aunc_{2m}[\OB\diamond \RB,\XX\oplus\YY]&\ge \frac{1}{2}\max\{\adom_{m}[\OB,\XX,\RB,\YY], \adom_{m}[\RB,\YY,\OB,\XX]\}.
\end{align*}
\end{lemma}

\begin{proof}
Proceed as in the proof of \cite{AABBL2021}*{Proposition 4.5}.
\end{proof}

The parameters $(\dom_m)_{m=1}^\infty$ also serve to estimate the conditionality parameters of Schauder bases, or biorthogonal systems, squeezed between two sequence spaces. If $\UU_1$ and $\UU_2$ are sequence spaces, we set
\[
\dom_m[\UU_1,\UU_2]=\dom_m[\EB,\UU_1,\EB,\UU_2], \quad m\in\NN.
\]
\begin{lemma}\label{lem:BHCC}
Let $\XB$ be a biorthogonal system of the Banach space $\XX$. Suppose that $\XB$   $\UU_1$-Hilbertian and $\UU_2$-Besselian. Then,
\[
\unc_m[\XB,\XX]\lesssim \dom_m[\UU_1,\UU_2].
\]
\end{lemma}
\begin{proof}
Let $C_1$ be the norm of the series transform, and $C_2$ be the norm of the coefficient transform. Pick $f\in\XX$ and $A\subseteq\NN$ with $\abs{A}\le m$. We have
\begin{align*}
\norm{S_A(f)}
&\le C_1 \norm{(\xx_n^*(f) \Ind_A(n) )_{n=1}^\infty}_{\UU_1}\\
&\le C_1 \dom_m[\UU_1,\UU_2] \norm{(\xx_n^*(f)  \Ind_A(n))_{n=1}^\infty}_{\UU_2}\\
&\le C_1 \dom_m[\UU_1,\UU_2] \norm{(\xx_n^*(f) )_{n=1}^\infty}_{\UU_2}\\
&\le C_1 C_2 \dom_m[\UU_1,\UU_2] \norm{f}.\qedhere
\end{align*}
\end{proof}

The following consequence of Lemma~\ref{lem:BHCC}, whose straightforward proof we omit, single out an argument that we will use several times.
\begin{corollary}\label{cor:OptConditions}
Let $\XB$ be a basis of a Banach space $\XX$, and let $\UU_1$ and $\UU_2$ be sequence spaces. Suppose that
Suppose that $\XB$   $\UU_1$-Hilbertian and $\UU_2$-Besselian, and that $\unc_m[\XB,\XX] \gtrsim \dom_m[\UU_1,\UU_2]$ for $m\in\NN$. Then,
\begin{enumerate}[label=(\roman*), leftmargin=*, widest=ii]
\item $\unc_m[\XB,\XX] \approx \dom_m[\UU_1,\UU_2]$ for $m\in\NN$; and
\item if $\UU_1'$ and $\UU_2'$ are sequence spaces with
\[
\inf_m  \frac{\dom_m[\UU_1',\UU_2']}{\dom_m[\UU_1,\UU_2]}=0,
\]
then either $\XB$ is not $\UU_1'$-Hilbertian or $\XB$ is not $\UU_2'$-Besselian.
\end{enumerate}
\end{corollary}

For further reference, we record the value of the parameters $\dom_m$ in some important cases. Given a sequence $f=(a_n)_{n=1}^\infty\in\FF^\NN$ and an increasing map $\pi\colon\NN\to\NN$, let $f_\pi=(b_n)_{n=1}^\infty\in\FF^\NN$ be the sequence defined by $b_n=a_k$ if $n=\pi(k)$ for some $k\in\NN$, and $b_n=0$ otherwise. A sequence space $\Sym$ is said to be \emph{subsymmetric} if $f_\pi\in\Sym$ for every $f\in \Sym$ and every increasing map $\pi\colon\NN\to\NN$, and we have $\norm{f_\pi}_\Sym=\norm{f}_\Sym$. In general, $\dom_m[\Sym_1,\Sym_2]=\adom_m[\Sym_1,\Sym_2]$ whenever $\Sym_1$ and $\Sym_2$ are subsymmetric sequence spaces. As $\ell_p$-spaces are concerned, we have
\begin{equation}\label{eq:Comlplq}
\dom_m[\ell_q,\ell_r]=\adom_m[\ell_q,\ell_r]=m^{1/q-1/r},\quad m\in\NN,\; 0< q \le r\le \infty.
\end{equation}
Note that $\ell_p$ is not locally convex space in the case when $0<p<1$. However, the parameters $\dom$ and $\adom$ still have sense in the nonlocally convex setting. Other subsymmetric sequence spaces of interest for us are Lorentz sequence spaces. Notice that
\begin{equation}\label{eq:LorentzNorm}
\norm{\sum_{n=1}^m \frac{1}{n^{1/p}} \ee_n}_{\ell_{p,q}}=H_m^{1/q}, \quad m\in\NN,\; 0<p<\infty,\; 0<q \le \infty,
\end{equation}
where $H_m$ is the $m$th harmonic number. Combining this identity with H\"older's inequality we obtain
\begin{equation}\label{eq:domLorentz}
\dom_m[\ell_{p,q},\ell_{p,r}]=H_m^{1/q-1/r}, \quad m\in\NN,\; 0<p<\infty,\; 0<q \le r \le \infty.
\end{equation}

We also record the weighted version of \eqref{eq:domLorentz}. Given $0<q\le\infty$ and a weight $\ww=(w_n)_{n=1}^\infty$ whose primitive weight $(s_n)_{n=1}^\infty$ is doubling, the Lorentz sequence space $d_{1,q}(\ww)$ is the quasi-Banach space consisting of all sequences $f\in c_0$ whose non-increasing rearrangement $(a_n)_{n=1}^\infty$ satisfies
\[
\norm{f}_{d_{1,q}(\ww)}=\left(\sum_{n=1}^\infty (s_n a_n)^q \frac{w_n}{s_n}\right)^{1/q}<\infty,
\]
with the usual modification in $q=\infty$. if $\ww=(n^{1/p-1})_{n=1}^\infty$, then $d_{1,q}(\ww)=\ell_{p,q}$ up to an equivalent norm. We have
\begin{equation}\label{eq:domWLorentz}
\dom_m[d_{1,q}(\ww),d_{1,r}(\ww)]=(H_m[\ww])^{1/q-1/r}, \; m\in\NN,\;  0<q \le r \le \infty,
\end{equation}
where $H_m[\ww]=\sum_{n=1}^m w_n/s_n$. If
\[
s_{rm}\ge 2 s_m, \quad m\in\NN,
\]
for some integer $r$, in which case we say that $(s_m)_{m=1}^\infty$ has the \emph{lower regularity property} (LRP for short),
the growth of the sequence $(H_m[\ww])_{m=1}^\infty$ can be computed.

\begin{lemma}\label{lem:equivalentnorm}
Let $\ww$ be a weight whose primitive sequence $(s_m)_{m=1}^\infty$ has the LRP, and let $0<q\le\infty$. Then,
\[
\norm{f}_{d_{1,q}(\ww)}\approx  \left(\sum_{n=1}^\infty (s_n a_n)^q \frac{1}{n}\right)^{1/q}
\]
for every $f\in c_0$ with non-increasing rearrangement $(a_n)_{n=1}^\infty$. In particular, $H_m[\ww]\approx H_m$ for $m\in\NN$.
\end{lemma}

\begin{proof}
Let $\ww'=(w_n')_{n=1}^\infty$ be the weight defined by $w_n'=s_n/n$, and let $(s_n')_{n=1}^\infty$ be its primitive sequence. Since $(s_n)_{n=1}^\infty$ has the LRP, $s_n\approx s_n'$ for $n\in\NN$. Consequently, $\norm{f}_{d_{1,q}(\ww)}\approx \norm{f}_{d_{1,q}(\ww')}$ for $f\in c_0$ (see, e.g., \cite{AABW2021}*{\S9.2}). Since $w_n'/s_n'\approx 1/n$ for $n\in\NN$, the desired equivalence of quasi-norms holds. To obtain the equivalence for $H_m[\ww]$, we apply the equivalence between quasi-norms with $q=1$ and
\[
a_n=\begin{cases} 1/s_n & \text{ if } n\le m, \\ 0 & \text{ if } n>m,\end{cases}
\]
where $m$ runs over $\NN$.
\end{proof}

We conclude  this preliminary section with another equivalence for the quasi-norms of weighted Lorentz sequence spaces.
\begin{lemma}[see \cite{AABW2021}*{Equation (8.3)}]\label{lem:equivalentnormbis}
Let $\ww$ be a weight whose primitive sequence $(s_n)_{n=1}^\infty$ is doubling, and let $0<q<\infty$. Then
\[
\norm{f}_{d_{1,q}(\ww)}\approx  \left(\sum_{n=1}^\infty  a_n^q (s_n^q-s_{n-1}^q)\right)^{1/q}
\]
for every $f\in c_0$ with non-increasing rearrangement $(a_n)_{n=1}^\infty$.
\end{lemma}

\section{Fourier coefficients of functions in $L_2(\TT, \abs{t}^{\lambda}\, dt)$, $\abs{\lambda}<1$.}\label{sect:Main}\noindent
With the aid of Proposition~\ref{lem:latticeEst}, we make our first move toward  the proof of Theorem~\ref{thm:Main}.

\begin{lemma}\label{lem:easyhalf}
Let $0\le \alpha<1$. Then,
\[
J_B[\TB,\HH_{-\alpha}]=[2,\infty]\quad  \text{and}\quad J_H[\TB,\HH_\alpha]=[1,2].
\]
\end{lemma}
\begin{proof}
By Proposition~\ref{lem:latticeEst}, it suffices to prove that $\TB$ is a $2$-Besselian basis of $\HH_{-\alpha}$, and a $2$-Hilbertian basis of $\HH_\alpha$. Taking into account that $\TB$ is a $2$-Besselian basis of $\HH_0$, the former assertion is a consequence of the embedding $\HH_{-\alpha}\subseteq \HH_0$. In turn, since  $\TB$  is a $2$-Hilbertian basis of $\HH_0$, the latter assertion follows from the embedding $\HH_0\subseteq \HH_\alpha$.
\end{proof}

The authors of  \cite{GW2014} computed the norm in $\HH_{\lambda}$, $-1<\lambda<1$, of the Dirichlet kernel $(D_m)_{m=0}^\infty$ defined by
\[
D_m=\sum_{n=-m}^m  \tau_n=\sum_{n=0}^{2m} \phi_m=\sum_{n=0}^{2m} \phi_m^\RR, \quad m\in\NN.
\]
For the reader's ease, we record this result of Garrig\'os and Wojtaszczyk that we will use a couple of times.

\begin{lemma}[see \cite{GW2014}*{Lemma 3.7}]\label{lem:Dirichlet}
Let $-1<\lambda<1$. Then
\[
\norm{D_m}_{\HH_\lambda}\approx m^{(1-\lambda)/2},  \quad m\in\NN.
\]
\end{lemma}

As a matter of fact, Lemma~\ref{lem:Dirichlet} provides valuable information on the Besselian and Hilbertian intervals of the trigonometric system in $\HH_\lambda$.

\begin{lemma}\label{lem:DirLB}
Let $0\le \alpha<1$. Then,
\[
J_B[\TB,\HH_\alpha]\subseteq[r_\alpha,\infty]\quad \text{and} \quad J_H[\TB,\HH_{-\alpha}]\subseteq[1,q_\alpha].
\]
\end{lemma}

\begin{proof}
Pick $s\in[1,\infty]$ and  suppose that $\TB$ is a $s$-Besselian basis of $\HH_\alpha$ (resp., a $s$-Hilbertian basis of $\HH_{-\alpha}$). Then,
\[
\sup_m \frac{\norm{ \sum_{n=-m}^m \ee_n}_s}{\norm{D_m}_{\HH_\alpha}}<\infty
\; \text{(resp.,}\;
\sup_m \frac{\norm{D_m}_{\HH_{-\alpha}}}{\norm{\sum_{n=-m}^m \ee_n}_s}<\infty \text{).}
\]
Since $\norm{\sum_{n=-m}^m \ee_n}_s\approx m^{1/s}$ for $m\in\NN$, we infer from Lemma~\ref{lem:Dirichlet} that
$1/s\le  (1-\alpha)/2$ (resp., $(1+\alpha)/2\le 1/s$).  Hence, $r_\alpha\le s$ (resp., $s\le q_\alpha$).
\end{proof}

To help the reader to grasp the issue of the optimality of Lemma~\ref{lem:DirLB}, we note that combining the embedding
\[
\HH_{\alpha} \subseteq L_q([-1/2,1/2]),\quad 0<\alpha<1, \;  1\le q<q_\alpha,
\]
which follows from H\"older's inequality, with Hausdorff-Young inequality, and taking into account that $r_\alpha$ and $q_\alpha$ are conjugate exponents, yields $(r_\alpha,\infty]\subseteq  J_B[\TB,\HH_\alpha]$. So, only whether the trigonometric system is a $r_\alpha$-Besselian basis of $\HH_{\alpha}$ is in doubt. To answer this question, we need to introduce some terminology. Let $\EE=\{ a\in\FF \colon \abs{A}=1\}$. The \emph{fundamental function} of  a basis $\XB=(\xx_n)_{n=1}^\infty$ of a Banach space $\XX$ is defined as
\[
\udf_m[\XB,\XX]= \sup\left\{  \norm{\Ind_{\varepsilon,A}[\XB]} \colon \varepsilon\in\EE^A,\; \abs{A}  \le m\right\},
\]
where, for $A\subseteq\NN$ finite and scalars $\varepsilon=(\varepsilon_n)_{n\in A}\in\EE^A$,
\[
\Ind_{\varepsilon,A}[\XB,\XX]=\sum_{n\in A} \varepsilon_n\, \xx_n.
\]

Theorem~\ref{thm:FE} below is the last step on our route toward proving Theorem~\ref{thm:Main}.

\begin{theorem}\label{thm:FE}
Let $0<\alpha<1$.
\begin{enumerate}[label=(\roman*),leftmargin=*, widest=ii]
\item The Fourier coefficient transform
\[
f\mapsto \widehat{f}=\left(\int_{-1/2}^{1/2} f(t) e^{-2\pi i  n t }\, dt\right)_{n=-\infty}^\infty
\]
is a bounded operator from $\HH_\alpha$ into $\ell_{r_\alpha,2}$.

\item\label{part:Hilbertian}The Fourier series transform
\[
(a_n)_{n=-\infty}^\infty\mapsto \sum_{n=-\infty}^\infty a_n \, \tau_n
\]
is a bounded operator from $\ell_{q_\alpha,2}$ into $\HH_{-\alpha}$.
\end{enumerate}
\end{theorem}

\begin{proof}
Let $-1<\lambda<1$. We have $\HH_\lambda^*=\HH_{-\lambda}$ via the dual pairing
\[
(f,g)\mapsto \int_{-1/2}^{1/2} f(t)g(t)\, dt, \quad f\in \HH_\lambda, \; g\in \HH_{-\lambda}.
\]
Moreover, $\TB$, regarded as a basis of $\HH_{-\lambda}$, is the dual basis of the own system $\TB$ regarded as a basis of $\HH_\lambda$. Consequently, by duality, it suffices to prove \ref{part:Hilbertian}. To that end, we pick $0<\alpha<\beta<1$. Given $A\subseteq\ZZ$ we set
\[
J_n[A]=\{1+\abs{n-k}\colon k\in A\},\quad n\in A.
\]
Since, for some constant $C_\beta$,
\begin{equation}\label{eq:GWEstimate}
\abs{\widehat{\ww_{-\beta}}(n)}\le  \frac{C_\beta}{(1+\abs{n})^{1-\beta}}, \quad n\in\ZZ,
\end{equation}
(see \cite{GW2014}*{Lemma A.2}), for every $m\in\NN$, every $A\subseteq\ZZ$  with $\abs{A}\le m $, and every $\varepsilon=(\varepsilon_n)_{n\in A}\in\TT^A$ we have
\begin{align*}
\norm{\Ind_{\varepsilon,A}[\TB,\HH_{-\beta}]}^2
&=\sum_{(n,k)\in A^2}\varepsilon_n \overline{\varepsilon_k} \widehat{\ww_{-\beta}}(k-n)\\
&\le C_\beta \sum_{n\in A} \sum_{k\in A} \frac{1}{(1+\abs{n-k})^{1-\beta}}\\
&\le 2 C_\beta \sum_{n\in A} \sum_{j\in J_n[A]} \frac{1}{j^{1-\beta}}\\
&\le 2  C_\beta m \sum_{j=1}^m \frac{1}{j^{1-\beta}}\\
&\le \frac{2C_\beta}{\beta} m^{1+\beta}.
\end{align*}
We have obtained
\[
\udf_m[\TB, \HH_{-\beta}]\le \sqrt{\frac{2 C_\beta} {\beta}} m^{1/q_\beta}, \quad m\in\NN.
\]
Therefore, by \cite{AABW2021}*{Corollary 9.13}, there is $0<r\le 1$ such that the Fourier series transform is a bounded operator from $\ell_{p_\beta,r}$ into $\HH_{-\beta}$. In fact, since $\HH_{-\beta}$ is a Banach space, $r=1$.

By orthogonality, the Fourier series transform is a bounded operator from  $\ell_2$ into $\HH_0$.  Pick $0<\theta<1$. By interpolation (by means of the real method)  the Fourier series transform is a bounded operator from $\ell_{q,2}$ into $\HH_{-\gamma}$, where
\[
\frac{1}{q}= \frac{\theta}{q_\beta} + \frac{ 1-\theta}{2}, \quad \gamma=\theta\beta.
\]
(see \cite{BerLof1976}*{Theorems 5.3.1 and 5.4.1}). If we choose $\theta=\alpha/\beta$, then $\gamma=\alpha$ and
\[
\frac{1}{q}= \frac{\alpha(1+\beta)}{2\beta} + \frac{ \beta-\alpha}{2\beta}=\frac{1+\alpha}{2},
\]
that is, $q=q_\alpha$.
\end{proof}

\begin{proof}[Proof of Theorem~\ref{thm:Main}]
By  Lemma~\ref{lem:gathered} and the identities \eqref{eq:cossin} and \eqref{eq:exp}, it suffices to prove the assertions involving the trigonometric system $\TB$. Since $\ell_p\subseteq \ell_{p,2}$ for $p\le 2$, and $\ell_{p,2}\subseteq \ell_p$ for $p\ge 2$, Theorem~\ref{thm:FE} gives that
$r_\alpha\in J_B[\TB,\HH_\alpha]$ and $q_\alpha\in J_H[\TB,\HH_{-\alpha}]$.  In light of Lemma~\ref{lem:DirLB} and Lemma~\ref{lem:easyhalf},  the proof is over.
\end{proof}

Theorem~\ref{thm:Main} says, in particular, that $q_\alpha$ is the optimal index $q$ such that the trigonometric system in $\HH_{-\alpha}$ if $q$-Hilbertian, and $r_\alpha$ is the optimal index $r$ such that the trigonometric system in $\HH_\alpha$ if $r$-Besselian. What remains of this section is devoted to proving that the estimates obtained in  Theorem~\ref{thm:FE} are also optimal in the `secondary' index. To that end, we  need to compute, up to equivalence, the Fourier coefficients of  the power weight $\ww_\lambda$, $-1<\lambda<0$.
\begin{lemma}[cf.\ Equation \eqref{eq:GWEstimate}]\label{lem:WeightCoef}
Let $0<\alpha<1$. Then,
\[
\widehat{\ww_{-\alpha}}(n)\approx \frac{1}{(1+\abs{n})^{1-\alpha}}, \quad n\in\ZZ.
\]
\end{lemma}

\begin{proof}
Pick $n\in\NN$. We have $\widehat{\ww_{-\alpha}}(-n)=\widehat{\ww_{-\alpha}}(n)$ and
\[
\widehat{\ww_{-\alpha}}(n)=2\int_0^{1/2} \frac{ \cos(2\pi n t) }{t^\alpha}\, dt=\frac{2^\alpha}{n^{1-\alpha}} A_n,
\]
where
\[
A_n=\int_0^{n} \frac{ \cos(\pi x) }{x^\alpha}\, dx= \sum_{k=1}^n (-1)^{k-1} \int_{k-1}^{k} \frac{\abs{\cos(\pi x)}}{ x^{\alpha}} \, dx.
\]
By Leibniz test for alternating series, $A_n>0$  for all $n\in\NN$, and there exists $\lim_n A_n\in(0,\infty)$.
\end{proof}

\begin{proposition}
Let $0<\alpha<1$,  $0<p<\infty$, and $0<q\le \infty$.
\begin{enumerate}[label=(\roman*),leftmargin=*, widest=ii]
\item The Fourier coefficient transform is a bounded operator from $\HH_\alpha$ into $\ell_{p,s}$ if and only if
$p>r_\alpha$, or $p=r_\alpha$ and $s\ge 2$.
\item The Fourier series transform is a bounded operator from $\ell_{p,s}$ into $\HH_{-\alpha}$ if and only if $ p< q_\alpha$, or $p= q_\alpha$ and $s\le 2$.
\end{enumerate}
\end{proposition}

\begin{proof}
Since $\ell_{p_1,s_1}\subseteq \ell_{p_2,s_2}$ if $p_1<p_2$, or $p_1=p_2$ and $s_1\le s_2$, in light of Theorem~\ref{thm:FE}  it suffices to prove that
\begin{itemize}[leftmargin=*]
\item if Fourier series transform is a bounded operator from $\ell_{p_\alpha,s}$ into $\HH_{-\alpha}$, then $s\le 2$; and that
\item if the coefficient transform is bounded from $\HH_\alpha$ into $\ell_{q_\alpha,s}$, then $s\ge 2$.
\end{itemize}

Since the former assertion can be deduced from the latter by duality, it suffices to prove the latter one.

Let $(f_m)_{m=1}^\infty$ be the sequence of trigonometric polynomials defined by
\[
f_m(t)=\sum_{n=1}^m \frac{1}{n^{1/q_\alpha}} e^{2\pi i n t}, \quad \frac{-1}{2} \le t \le \frac{1}{2}, \; m\in\NN.
\]
By Lemma~\ref{lem:WeightCoef}, for $m\in\NN$ we have
\begin{align*}
\norm{f_m}^2_{\HH_{-\alpha}}
&=\sum_{1\le n,k\le m} \frac{\widehat{\ww_{-\alpha}}(k-n)}{(nk)^{(1+\alpha)/2}}
\approx \sum_{n=1}^m\sum_{k=1}^n \frac{(1+n-k)^{\alpha-1}}{ (nk)^{(1+\alpha)/2}}\\
&=\sum_{n=1}^m \frac{A_n}{n^{(1+\alpha)/2}},\\
\end{align*}
where
\[
A_n=\sum_{k=1}^n  \frac{1}{k^{(1+\alpha)/2}} \frac{1}{(1+n-k)^{1-\alpha}}.
\]
Set
\begin{align*}
B_n&=\int_0^n \frac{dx}{x^{(1+\alpha)/2} (n-x)^{1-\alpha}}, \; n\in\NN, \; \text{and}\\
R_n&=-\frac{1}{n^{1-\alpha}}-\frac{1}{n^{(1+\alpha)/2}} +\frac{2/(1-\alpha)}{(n-1)^{1-\alpha}}+\frac{1/\alpha}{(n-1)^{(1+\alpha)/2}}, \; n\in\NN,\; n\ge 2.
\end{align*}
We have
\begin{align*}
B_n&=\frac{\beta((1-\alpha)/2,\alpha)}{n^{(1-\alpha)/2}},\; n\in\NN,\\
A_n&\le B_n\le A_n+R_n, \; n\in\NN, \; n\ge 2,\; \text{and}\\
0&=\lim_n  n^{(1-\alpha)/2} R_n.
\end{align*}
We infer that $A_n\approx n^{-(1-\alpha)/2}$ for $n\in\NN$. Consequently,
\begin{equation}\label{eq:SecondaryEstimate}
\norm{f_m}_{\HH_{-\alpha}}\approx H_m^{1/2},\quad m\in\NN.
\end{equation}

Suppose that the Fourier series transform is a bounded operator from $\ell_{q_\alpha,s}$ into $\HH_{-\alpha}$. Combining
\eqref{eq:SecondaryEstimate} with \eqref{eq:LorentzNorm} gives $s\le 2$.
\end{proof}

\section{Conditionality parameters of Schauder bases}\label{sect:CPSchauder}\noindent
The trigonometric system $\TB$, regarded as a sequence in $\HH_\lambda$, $0<\abs{\lambda}<1$,  is the first example of a conditional Schauder basis of a Hilbert space arisen in the literature (see \cite{Gelbaum1951}). In hindsight, that  $\TB$ is not unconditional can be deduced from combining the papers \cite{KotheToeplitz1934}, where it is proved that every semi-normalized unconditional basis of $\ell_2$ is equivalent to its unit vector system, and \cite{Altman1949}, where it is proved that  $\TB$, regarded as a sequence in $\HH_\lambda$, is not equivalent to the unit vector system of $\ell_2$. Notice that the last-mentioned result can be deduced from Theorem~\ref{thm:Main}, and also from Lemma~\ref{lem:Dirichlet}. So, It shouldn't be surprising that Theorem~\ref{thm:Main} enables us to move forward with the theory of conditional bases.
\begin{theorem}\label{thm:AA}
For each $1<q\le 2 \le r<\infty$ there is a $q$-Hilbertian $r$-Besselian Schauder basis  $\XB$ of $\ell_2$ with
\[
\unc_m[\XB,\ell_2]\approx \aunc_m[\XB,\ell_2] \approx m^{1/q-1/r}, \quad m\in\NN.
\]
Moreover, $\XB$ is not $\ell_{q_1}$-Hilbertian for any $q_1>q$ nor $\ell_{r_1}$-Besselian for any $r_1<r$.
\end{theorem}

\begin{proof}
Pick $0<\alpha<1$ with $r_\alpha=r$, and $0<\beta<1$ with $q_\beta=q$. By Theorem~\ref{thm:Main}, Lemma~\ref{lem:gathered}, and Lemma~\ref{lem:ccdom}, the rotated system $\XB:=\TB^{\,\RR} \diamond \TB^{\,\RR}$ is a $p$-Hilbertian $q$-Besselian Schauder basis of the Hilbert space $\HH:=\HH_{-\beta}\oplus \HH_\alpha$ with
\[
2 \aunc_{2m} [\XB,\HH] \ge d_m:=\adom_m[ \TB^{\,\RR},\HH_{-\beta},\TB^{\,\RR},\HH_\alpha], \quad m\in\NN.
\]
In turn, by Lemma~\ref{lem:Dirichlet},
\[
d_{2m+1}\ge \frac{\norm{D_m}_{\HH_{-\beta}}} {\norm{D_m}_{\HH_\alpha}}\approx m^{(1+\beta)/2-(1-\alpha)/2}= m^{1/q-1/r}, \quad m\in\NN.
\]
We infer that $\aunc_{m}[\XB]\gtrsim m^{1/q-1/r}$ for $m\in\NN$. In light of \eqref{eq:Comlplq}, an application of Corollary~\ref{cor:OptConditions} puts an end to the proof.
\end{proof}

We will derive Theorem~\ref{thm:CPlp} from the first of the two consequences of Theorem~\ref{thm:AA} that we record  below. The second one will be used in Section~\ref{sect:CPAG}.

\begin{corollary}\label{cor:CPlp}
Let $\XX$ be a Banach space with a Schauder basis $\BB$, and let $0\le \alpha_0<1$. Suppose that $\unc_m[\BB,\XX]\lesssim m^{\alpha_0}$ for $m\in\NN$, and that $\XX$ contains a complemented subspace isomorphic to $\ell_p$ for some $1<p<\infty$. Then, for each $\alpha\in[\alpha_0,1)$, there is a  Schauder basis  $\XB$ of $\XX$ with $\unc_m[\XB,\XX]\approx m^{\alpha}$ for $m\in\NN$.
\end{corollary}

\begin{corollary}\label{cor:CPlpBis}
Let $\XX$ be a Banach space with a Schauder basis $\BB$, and let $1<p<\infty$. Let $1<q_0\le \min\{2, p\}$ and  $\max\{2,p\} \le r_0<\infty$. Suppose that
\begin{itemize}[leftmargin=*]
\item $\BB$ is $q_0$-Hilbertian and $r_0$-Besselian,
\item $\unc_m[\BB,\XX]\lesssim m^{1/q_0-1/r_0}$ for $m\in\NN$, and
\item $\XX$ contains a complemented subspace isomorphic to $\ell_p$.
\end{itemize}
Then, for each $1<q\le q_0$ and each  $r_0\le r <\infty$, there is a Schauder basis  $\XB$ of $\XX$ such that
\begin{enumerate}[label=(\roman*), leftmargin=*,widest=iii]
\item $\unc_m[\XB,\XX]\approx \aunc_m[\XB,\XX] \approx m^{1/q-1/r}$ for $m\in\NN$,
\item $\XB$ is $q$-Hilbertian $r$-Besselian, and
\item  $\XB$ is not $\ell_{q_1}$-Hilbertian for any $q_1>q$ nor $\ell_{r_1}$-Besselian for any $r_1<r$.
\end{enumerate}
\end{corollary}

Before proceeding with the proof these results, it will be convenient to introduce some notation. Given a sequence $\XB=(\xx_n)_{n=1}^\infty$ in a Banach space $\XX$ we set
\[
\XB^{(m)}=(\xx_n)_{n=1}^m, \quad m\in\NN.
\]
If $(\XX_j)_{j=1}^\infty$ are Banach spaces, we denote by $L_j$ the canonical embedding of $\XX_j$ into $\prod_{j=1}^\infty\XX_j$, $j\in\NN$. For each $j\in\NN$, let $\XB_j=(\xx_{j,n})_{n=1}^{m_j}$ be a finite family in the Banach space $\XX_j$. We denote by $\bigoplus_{j=1}^\infty \XB_j=(\yy_k)_{k=1}^\infty$ the natural arrangement of the family
\[
\overline{\xx}_{j,n}:=L_j(\xx_{j,n}), \quad j\in\NN,\;  1\le n \le m_j,
\]
that is, $\yy_k=\overline{\xx}_{j,n}$ if $k=\sum_{i=1}^{j-1}m_i +n$. Given a Schauder basis $\XB$ and a sequence $\mm=(m_j)_{j=1}^\infty$ in $\NN$ we set
\[
\XB^{(\mm)}= \bigoplus_{j=1}^\infty \XB^{(m_j)}.
\]
\begin{lemma}\label{lem:Infinite}
Let $\XB=(\xx_j)_{j=1}^\infty$ be a Schauder basis of a Banach space $\XX$, $\mm=(m_j)_{j=1}^\infty$ be a sequence in $\NN$, and $\UU$ be a sequence space. Set
\[
\YY:=\left(\bigoplus_{j=1}^\infty [\xx_n \colon 1\le n \le m_j]\right)_\UU.
\]
\begin{enumerate}[label=(\roman*), leftmargin=*,widest=iii]
\item\label{InfiniteSchauder}  $\XB^{(\mm)}$ is a Schauder basis of $\YY$.
\item\label{InfiniteUNC} If $\mm$ is unbounded, then $\unc_m[\XB^{(\mm)},\YY]=\unc_m[\XB,\XX]$.
\item\label{InfiniteHB} Let $1\le q\le\infty$. Suppose that both $\XB$ and the unit vector system of $\UU$ are $q$-Hilbertian (resp., $q$-Besselian). Then, $\XB^{(\mm)}$ is $q$-Hilbertian (resp., $q$-Besselian)  regarded as a basis of $\YY$.
\item\label{InfiniteAUNC} Suppose that $\sum_{i=1}^{j-1} m_i \lesssim m_j$ for $j\in\NN$. Let $\delta\colon(0,\infty)\to(0,\infty)$ be a doubling function such that $\aunc_m[\XB,\XX] \gtrsim \delta(m)$ for $m\in\NN$. Then,  $\aunc_m[\XB^{(\mm)},\YY] \gtrsim \delta(m)$ for $m\in\NN$.
\end{enumerate}
\end{lemma}

\begin{proof}
Parts~\ref{InfiniteSchauder}  and \ref{InfiniteUNC} are straightforward, and \ref{InfiniteHB} follows from the natural isometry between $(\bigoplus_{j=1}^\infty \ell_q^{m_j})_q$ and $\ell_q$.   To prove \ref{InfiniteAUNC}, proceed as in the proof of \cite{AADK2018}*{Lemma 2.3}.
\end{proof}

\begin{proof}[Proof of Corollaries~\ref{cor:CPlp} and \ref{cor:CPlpBis}]
To prove  Corollary~\ref{cor:CPlp}, we choose $1<q\le 2\le r<\infty$ such that $\alpha=1/q-1/r$. Then, both to prove Corollary~\ref{cor:CPlp} and  Corollary~\ref{cor:CPlpBis}, we consider the Schauder basis $\XB_{q,r}$  of $\ell_2$ provided by Theorem~\ref{thm:AA}. For each $m\in\NN$, let $\HH_m$ be the subspace of $\ell_2$ spanned by $\XB_{q,r}^{(m)}$. Set $\mm=(2^j)_{j=1}^\infty$. By \eqref{eq:UNCCDS}, and \eqref{eq:Comlplq}, Lemma~\ref{lem:Infinite} and Corollary~\ref{cor:OptConditions}, the sequence $\BB\oplus \XB^{(\mm)}$ is a Schauder basis of $\YY:=\XX\oplus\left( \bigoplus_{n=1}^\infty \HH_{2^n} \right)_p$ satisfying the desired properties. We have
\[
\left( \bigoplus_{n=1}^\infty \HH_{2^n} \right)_p\equiv \left( \bigoplus_{n=1}^\infty \ell_2^{2^n} \right)_p\approx \ell_p
\]
(see \cite{Pel1960}*{Proof of Theorem 7}). In turn, since  $\ell_p$ is isomorphic to its square, $\XX\oplus\ell_p\simeq\XX$. We infer that $\YY\simeq\XX$.
\end{proof}

\begin{proof}[Proof of Theorem~\ref{thm:CPlp}]
Apply Corollary~\ref{cor:CPlp} in the case where $\XX$ is $\ell_p$ and $\BB$ is its unit vector system, so that $\alpha_0=0$.
\end{proof}

\section{Conditionality parameters of almost greedy bases}\label{sect:CPAG}\noindent
Let us draw reader's attention to the existence of  an almost greedy counterpart of Theorem~\ref{thm:EJ}. Namely, the authors of \cite{ABW2021} proved the following result.
\begin{theorem}[see \cite{ABW2021}*{Theorems 1.1 and 3.16}]\label{thm:ABW}
Let $\XB$ be an almost greedy basis of a superreflexive Banach space, and let $\ww$ be the weight whose primitive sequence is the fundamental function $(\udf_m)_{m=1}^\infty$ of $\XB$. Then,  there are $1<r\le s<\infty$ such that $\XB$ is $d_{1,r}(\ww)$-Hilbertian and $d_{1,s}(\ww)$-Besselian. Moreover, for every $1<t<\infty$, $d_{1,t}(\ww)$ is a superreflexive Banach space.
\end{theorem}
In light of Theorem~\ref{thm:ABW}, Lemma~\ref{lem:BHCC}, and Equation~\eqref{eq:domWLorentz}, looking for embeddings involving Lorentz sequence spaces is a reasonable way of obtaining upper estimates for the conditionality parameters of almost greedy bases. Here, we focus on  almost greedy bases arising from the so called \emph{DKK-method} invented in \cite{DKK2003}.

Let $(\XX,\norm{\cdot}_\XX)$  be a Banach space with a semi-normalized Schauder basis $\XB=(\xx_n)_{n=1}^\infty$, and let $(\Sym,\norm{\cdot}_\Sym)$ be a symmetric or, more generally, subsymmetric sequence space. Let $(\Lambda_m)_{m=1}^\infty$ be the fundamental function of the unit vector system of $\Sym$, that is,
\[
\Lambda_m=\norm{\sum_{j=1}^m \ee_j}_\Sym, \quad m\in\NN.
\]
Let $\sigma=(\sigma_n)_{n=1}^\infty$ be an \emph{ordered partition} of $\NN$, i.e., a partition into integer intervals with
\[
\max(\sigma_n)<\min(\sigma_{n+1}),\quad n\in\NN.
\]
The \emph{averaging projection} $P_\sigma\colon\FF^\NN\to\FF^\NN$ associated with the ordered partition $\sigma$ can be expressed as
\[
P_\sigma(f)=\sum_{n=1}^\infty \langle \vv_n^*,f\rangle \, \vv_n,
\]
where
\begin{equation}\label{eq:18}
\quad \vv_n= \frac{1}{\Lambda_{\abs{\sigma_n}}}  \sum_{j\in\sigma_n} \ee_j,\quad
\vv_n^*=  \frac{\Lambda_{\abs{\sigma_n}}}{\abs{\sigma_n}}   \sum_{j\in\sigma_n} \ee_j,
\quad n\in\NN,
\end{equation}
and $\langle\cdot,\cdot\rangle$ is the natural dual pairing defined by
\[
\langle f, g\rangle = \sum_{n=1}^\infty a_n b_n, \quad f=(a_n)_{n=1}^\infty\in\FF^\NN,\; g=(b_n)_{n=1}^\infty\in\FF^\NN,\; fg\in\ell_1.
\]

Let $Q_\sigma=\Id_{\FF^\NN}-P_\sigma$ be the complementary projection. We define  $\norm{\cdot }_{\XB,\Sym,\sigma}$ on $c_{00}$ by
\[
\norm{f}_{\XB,\sigma,\Sym}=\norm{Q_\sigma(f)}_\Sym+\norm{\sum_{n=1}^\infty \vv_n^*(f)\, \xx_n}_\XX, \quad f\in c_{00}.
\]
The completion of the normed space $(c_{00},\norm{\cdot }_{\XB,\Sym,\sigma})$ will be denoted by $\YY[\XB,\Sym,\sigma]$. This method for building Banach spaces was invented in \cite{DKK2003}. The authors of \cite{AADK2018} delved into its study and named it  the DKK-method. For the purposes of this paper, it will be necessary to achieve some new properties of the unit vector system of $\YY[\XB,\Sym,\sigma]$.  To properly enunciate them, we introduce some terminology.  Let us regard the functionals on $\Sym$ as sequences acting on $\Sym$ through the natural dual pairing. With this convention, we have $c_{00} \subseteq \Sym^*\subseteq \ell_{\infty}$,  and the closed subspace of $\Sym^*$ spanned by $c_{00}$ is a subsymmetric sequence space that we denote by $\Sym_0^*$. Let $(\Gamma_m)_{m=1}^\infty$ be the fundamental function of the unit vector system of $\Sym_0^*$, that is,
\[
\Gamma_m=\norm{\sum_{j=1}^m \ee_j}_{\Sym^*}, \quad m\in\NN.
\]
By \cite{LinTza1977}*{Proposition 3.a.6},
\begin{equation}\label{eq:FFSubSym}
\frac{m}{\Lambda_n} = c_m \Gamma_m,\quad 1\le c_m\le 2,\;  \quad m\in\NN.
\end{equation}

Recall that the dual basis $\YB^*$ of a Schuader basis $\YB$ of a Banach space $\YY$ is a Schauder basis of the Banach space it spans in $\YY^*$, and  $\YB^{**}$ is equivalent to $\YB$ (see \cite{AlbiacKalton2016}*{Corollary 3.2.4}). In particular, we have $(\Sym_0^*)^*_0=\Sym$ up to an equivalent norm.

Note also that the operator $P_\sigma$ is self-adjoint, i.e.,
\begin{equation*}
\langle f,P_\sigma(g) \rangle =\langle P_\sigma(f), g \rangle, \quad f,g\in\FF^\NN,\; fg\in\ell_1.
\end{equation*}
Consequently, $Q_\sigma$ also is selft-adjoint.

\begin{proposition}\label{eq:DualDKK}
Let $\XB$ be a Schauder basis of a Banach space $\XX$, $\Sym$ be a subsymmetric sequence space, and $\sigma=(\sigma_n)_{n=1}^\infty$ be an ordered partition of $\NN$.  Then, the unit vector system $\EB$ is a Schauder basis of $\YY[\XB,\Sym,\sigma]$ whose dual basis is equivalent to the unit vector system of $\YY[\BB,\Sym_0^*,\sigma]$, where $\BB$ is a suitable perturbation of $\XB^*=(\xx_n^*)_{n=1}^\infty$, that is, $\BB=(c_n\, \xx_n^*)_{n=1}^\infty$ for some semi-normalized sequence $(c_n)_{n=1}^\infty$ in $(0,\infty)$.
\end{proposition}

\begin{proof}
By \cite{AADK2018}*{Theorem 3.6}, the mapping
\[
(g,x)\mapsto g+ \sum_{n=1}^\infty \xx_n^*(x)\, \vv_n, \quad g\in c_{00}\cap Q_\sigma(\Sym),\; x\in\XX,\;\abs{\supp(x)}<\infty.
\]
defines an isomorphism from $Q_\sigma(\Sym)\oplus \XX$ onto $\YY[\XB,\Sym,\sigma]$. In turn, since $P_\sigma$ and $Q_\sigma$ are complementary linear bounded projections on $\Sym$
(see \cite{LinTza1977}*{Proposition 3.a.4}), the map
\[
u^*\mapsto (u^*(Q_\sigma(\ee_n))_{n=1}^\infty
\]
defines an isomorphism from $(Q_\sigma(\Sym))^*$ onto
\[
\UU:=\{ f \in\Sym^* \colon P_\sigma(f)=0\}.
\]

We infer that, if we regard the functionals in  $(\YY[\XB,\Sym,\sigma])^*$ as sequences acting on $\YY[\XB,\Sym,\sigma]$ by means of the natural dual pairing $\langle \cdot,\cdot\rangle$, the mapping
\begin{equation}\label{eq:17}
f\mapsto T(f):=\left( (\langle f, Q_\sigma(\ee_n )\rangle)_{n=1}^\infty, \sum_{n=1}^\infty \langle f,\vv_n\rangle \, \xx_n^*\right)
\end{equation}
defines an isomorphism from $(\YY[\XB,\Sym,\sigma])^*$ onto $\UU\times \XX^*$. In \eqref{eq:17}, we use the convention that $x^*=\sum_{n=1}^\infty a_n\, \xx_n^*$ means that $x^*\in\XX^*$ and $(a_n)_{n=1}^\infty\in\FF^\NN$ satisfy $\xx^*(x)=\sum_{n=1}^\infty a_n\, \xx_n^*(x)$ for every $x\in\XX$ finitely supported.

Since $Q_\sigma$ is selft-adjoint,
\[
(\langle f, Q_\sigma(\ee_n )\rangle)_{n=1}^\infty= (\langle Q_\sigma(f), \ee_n\rangle)_{n=1}^\infty=Q_\sigma(f),\quad f\in\FF^\NN.
\]

If $(\uu_n^*)_{n=1}^\infty$ are the vectors defined as in \eqref{eq:18} corresponding to the subsymmetric space $\Sym_0^*$, and $(c_n)_{n=1}^\infty$ is as in \eqref{eq:FFSubSym},  then $\vv_n =c_n \uu_n^*$ for all $n\in\NN$. Summing up, we have
\[
T(f)=\left( Q_\sigma(f),  \sum_{n=1}^\infty \langle \uu_n^*,f\rangle \, c_n\, \xx_n^*\right), \quad f\in (\YY[\XB,\Sym,\sigma])^*.
\]

Let $\XX^*_0$ denote the subspace of $\XX^*$ spanned by $\XB^*$.  Applying \cite{AADK2018}*{Theorem 3.6} with the Schauder basis $\BB=(c_n\, \xx_n^*)_{n=1}^\infty$ of $\XX^*_0$, the subsymmetric sequence space $\Sym_0^*$, and the partition $\sigma$  gives that the mapping
\[
f \mapsto \left( Q_\sigma(f),\sum_{n=1}^\infty \langle \uu_n^*, f\rangle \, c_n \, \xx_n^*\right), \quad f\in c_{00},
\]
extend to an isomorphism from $\YY[\BB^*,\Sym_0^*,\sigma]$ onto $Q_\sigma(\Sym^*_0)\oplus \XX^*_0$. We infer that $\YY[\BB^*,\Sym_0^*,\sigma]$ is, up to an equivalent norm, the closed subspace of $(\YY[\XB,\Sym,\sigma])^*$ spanned by $c_{00}$.
\end{proof}

\begin{lemma}\label{lem:decreasing}
Suppose that a sequence space $\UU$ satisfies an upper $q$-estimate, $1\le q<\infty$, and let $(\udf_m)_{m=1}^\infty$ denote the fundamental function of its unit vector system. Then, $(\udf_m^q/m)_{m=1}^\infty$ is essentially decreasing.
\end{lemma}

\begin{proof}
We have
\[
\udf_{km} \le C k^{1/q} \udf_m, \quad k,\,m\in\NN.
\]
where $C$ is upper $q$-estimate constant of $\UU$. Given $n$, $m\in\NN$ with $n>m$, we pick $k\in\NN$ with $km<n\le (k+1)m$. Then,
\[
\frac{\udf_{n}^q}{n}
\le \frac{\udf_{(k+1)m}^q}{km}\le C^r \frac{k+1}{k} \frac{\udf_{m}^q}{m}
\le 2 C^q  \frac{\udf_{m}^q}{m}.\qedhere.
\]
\end{proof}

Besides the LRP, we will use the \emph{upper regularity property} (URP for short) of sequences of positive scalars. We say that a sequence $(s_m)_{m=1}^\infty$ has the URP for short if there is $r\in\NN$ such that
\[
s_{rm}\le \frac{1}{2} r s_m, \quad m\in\NN.
\]

\begin{lemma}\label{lem:EstimatesRP}
Suppose that a sequence space $\UU$ satisfies an upper $q$-estimate for some $q>1$ (resp., a lower $r$-estimate for some $r<\infty$). Suppose also that the unit vector system of $\UU$ is a democratic basis. Then, its fundamental function has the URP (resp., the LRP).
\end{lemma}

\begin{proof}
Just combine \cite{LinTza1979}*{Propositions 1.f.3 and 1.f.7} with \cite{DKKT2003}*{Proposition 4.1}.
\end{proof}

\begin{lemma}\label{lem:SSHilbertian}
Suppose that a sequence space $\UU$ satisfies an upper $q$-estimate, $1\le q<\infty$, and let $\ww$ be the weight whose primitive sequence is the fundamental function of the unit vector system of $\UU$. Then, $d_{1,q}(\ww)\subseteq \UU$ continuously.
\end{lemma}

\begin{proof}
Pick $f=(a_n)_{n=1}^\infty\in c_{00}$. Put $t=\max_n\abs{a_n}$, and for each $k\in\NN$ consider the set
\[
E_k=\{n\in\NN \colon t 2^{-k}<\abs{a_n} \le t 2^{-k+1}\}.
\]
Notice that $(E_k)_{k=1}^\infty$ is a partition of  $\{n\in\NN \colon a_n\not=0\}$. Set $n_k=\abs{E_k}$ ($n_0=0$) and $m_k=\sum_{j=1}^k n_j$, so that, if $(b_n)_{n=1}^\infty$ is the non-increasing rearrangement of $f$,
\[
\{n\in\NN \colon t 2^{-k}< b_n \le t 2^{-k+1}\}
=\{n\in\NN \colon 1+m_{k-1}\le n \le m_k\}.
\]
For $n\in\NN$, let $s_n=\udf_n[\EB,\UU]$ and $w_n'=s_n^r-s_{n-1}^r$.  Let $C$ be the upper $C$-estimate constant of $\UU$. Using Abel's summation formula gives
\begin{align*}
\norm{f}^q &=\norm{\sum_{k=1}^\infty \sum_{n\in E_k} a_n \, \ee_n}^q\\
&\le  C^q \sum_{k=1}^\infty \norm{\sum_{n\in E_k} a_n \,\ee_n}^q\\
&\le C^q \sum_{k=1}^\infty  (t 2^{-k+1} s_{m_k})^q \\
&=  (2tC)^q \sum_{k=1}^\infty 2^{-kq} \sum_{j=1}^k ( s_{m_j}^q-s_{m_{j-1}}^q)\\
&= \frac{ (2tC)^q}{1-2^{-q}} \sum_{j=1}^\infty 2^{-jq} ( s_{m_j}^q-s_{m_{j-1}}^q)\\
&= \frac{ (4C)^q}{2^q-1} \sum_{j=1}^\infty  \sum_{n=1+m_{j-1}}^{m_j} ( t2^{-j})^q ( s_{n}^q-s_{n-1}^q)\\
&\le \frac{ (4C)^q}{2^q-1} \sum_{n=1}^\infty b_n^q ( s_{n}^q-s_{n-1}^q).
\end{align*}
Applying Lemma~\ref{lem:equivalentnormbis} puts an end to the proof.
\end{proof}

\begin{proposition}\label{prop:DKKHilbertian}
Let $\XB$ be a Schauder basis of a Banach space $\XX$, $\Sym$ be a subsymmetric sequence space, and $\sigma=(\sigma_n)_{n=1}^\infty$ be an ordered partition of $\NN$ with
\[
\sum_{n=1}^{m-1} \abs{\sigma_n}\le D \abs{\sigma_m}, \quad m\in \NN,
\]
for some $D\in\NN$.  Let $1<q<\infty$. Suppose that $\Sym$ is satisfies an upper $q$-estimate, that $\XB$ $q$-Hilbertian, and that $(\Lambda_m)_{m=1}^\infty$ has the LRP. Then, the unit vector system of $\YY[\XB,\Sym,\sigma]$ is  $d_{1,q}(\ww)$-Hilbertian, that is,
\[
d_{1,q}(\ww)\subseteq \YY[\XB,\Sym,\sigma]
\]
continuously, where $\ww=(\Lambda_n-\Lambda_{n-1})_{n=1}^\infty$.
\end{proposition}

\begin{proof}
Taking into account that $Q_\sigma$ is bounded on $\Sym$, we infer from Lemma~\ref{lem:SSHilbertian} that there is a constant $C_0$ such that
\[
\norm{Q_{\sigma}(f)}_\Sym\le C_0 \norm{f}_{d_{1,q}(\ww)}, \quad f\in c_{00}.
\]

By Lemma~\ref{lem:decreasing}, there is a constant $C_1$ such that
\[
\frac{1}{C_1} \frac{\Lambda_n^q}{n}
\le T_m:=\inf_{j\le m} \frac{\Lambda_j^q}{j} ,\quad n\ge m.
\]
Therefore, if $n\in\NN$ and  $j\in\sigma_n$,
\[
\frac{\Lambda^q_{\abs{\sigma_n}}}{\abs{\sigma_n}}
\le (D+1) \frac{\Lambda^q_{(D+1)\abs{\sigma_n}}}{(D+1)\abs{\sigma_n}}
\le (D+1)C_1 T_j.
\]

By Lemma~\ref{lem:equivalentnorm}, there is a constant $C_2$ such that
\[
\sum_{j=1}^\infty \frac{\Lambda_j^q b_j^q}{j} \le C_2^q \norm{f}_{d_{1,q}(\ww)}^q,\; (b_j)_{j=1}^\infty\; \text{non-increasing rearrangement of}\; f.
\]
Let $C_3$ be the norm of the series transform with respect to $\XB$, regarded as an operator from $\ell_q$ into $\XX$. Pick $f=(a_j)_{j=1}^\infty\in c_{00}$, and let  $(b_j)_{j=1}^\infty$ be its non-increasing rearrangement. Using H\"older's inequality and the rearrangement inequality we obtain
\begin{align*}
\norm{\sum_{n=1}^\infty \langle \vv_n^*,f\rangle\, \xx_n}^q
&\le C_3^q \sum_{n=1}^\infty \abs{\langle \vv_n^*,f\rangle}^q \\
&\le C_3^q \sum_{n=1}^\infty \frac{\Lambda^q_{\abs{\sigma_n}}}{\abs{\sigma_n}} \sum_{j\in\sigma_n} \abs{a_j}^q \\
&\le (D+1) C_1^qC_3^q\sum_{j=1}^\infty T_j \abs{a_j}^q \\
&\le (D+1) C_1^qC_3^q \sum_{j=1}^\infty T_j b_j^q \\
&\le (D+1) C_1^qC_3^q \sum_{j=1}^\infty \frac{\Lambda_j^q}{j} b_j^q \\
&\le (D+1) C_1^qC_2^qC_3^q  \norm{f}_{d_{1,q}(\ww)}^q.
\end{align*}

Summing up, we obtain $\norm{f}_{\YY[\XB,\Sym,\sigma]}\le C \norm{f}_{d_{1,q}(\ww)}$, where 
\[
C=C_0+(1+D)^{1/q} C_1 C_2 C_3.\qedhere
\]
\end{proof}

\begin{proposition}\label{prop:DKKBesselian}
Let $\XB$ be a Schauder basis of a Banach space $\XX$, $\Sym$ be a subsymmetric sequence space, and $\sigma=(\sigma_n)_{n=1}^\infty$ be an ordered partition of $\NN$ with
\begin{equation}\label{eq:OPConditionA}
\sum_{n=1}^{m-1} \abs{\sigma_n}\lesssim \abs{\sigma_m}, \quad m\in \NN.
\end{equation}
Let $1<r<\infty$. Suppose  that $\Sym$ satisfies a lower $r$-estimate, that $\XB$ is $r$-Besselian, and that $(\Lambda_m)_{m=1}^\infty$ has the URP. Then, the unit vector system of $\YY[\XB,\Sym,\sigma]$ is $d_{1,r}(\ww)$-Besselian, that is,
\[
\YY[\XB,\Sym,\sigma] \subseteq d_{1,r}(\ww)
\]
continuously, where $\ww=(\Lambda_n-\Lambda_{n-1})_{n=1}^\infty$.
\end{proposition}

\begin{proof}
Let $\ww^*$ be the weight whose primitive sequence is $(\Gamma_m)_{m=1}^\infty$. By \cite{LinTza1979}*{Proposition 1.f.r}, $\Sym_0^*$ satisfies an upper $r'$-estimate. In turn, by duality, $\XB^*$ is an $r'$-Hilbertian. Hence, any perturbation $\YB$ of $\XB^*$ is  $r'$-Hilbertian. An application of Proposition~\ref{prop:DKKHilbertian} gives
\begin{equation}\label{eq:DualEmb}
d_{1,r'}(\ww^*)\subseteq  \YY[\YB,\Sym_0^*,\sigma].
\end{equation}
By Lemma~\ref{lem:EstimatesRP}, $(\Lambda_n)_{n=1}^\infty$ has the LRP. Consequently, by \cite{ABW2021}*{Theorem  3.10} and Equation \eqref{eq:FFSubSym}, $d_{1,r'}(\ww^*)$ is, via the natural dual pairing, and up to an equivalent norm, the dual space of $d_{1,r}(\ww)$. In turn, by Proposition~\ref{eq:DualDKK}, we can choose $\YB$ so that $\YY[\YB,\Sym_0^*,\sigma]$ is, via the natural dual pairing, and up to an equivalent norm, the dual space of $\YY[\XB,\Sym,\sigma]$. Therefore, dualizing \eqref{eq:DualEmb} yields the desired embedding.
\end{proof}

Our next result must be regarded as a paradigm to prove the existence of almost greedy bases whose conditionality parameters grow as a given sequence.

\begin{lemma}\label{lem:BB}
Let $\XX$ be a Banach space with a Schauder basis $\BB$, Let $\Sym$ be a subsymmetric sequence space, and let  $1<q\le r<\infty$. Suppose that  $\BB$ is both $q$-Hilbertian and $r$-Hilbertian, and that  $\aunc_m[\BB,\XX] \gtrsim m^{1/q-1/r}$ for $m\in\NN$. Suppose also that $\Sym$ satisfies an upper $q$-estimate and a lower $r$-estimate, and that it is complemented in $\XX$. Let $\ww$ be a weight whose primitive sequence is equivalent to the fundamental function $(\Lambda_m)_{m=1}^\infty$ of $\Sym$. Then, $\XX$ has an almost greedy basis $\XB$ such that
\begin{enumerate}[label=(\roman*),leftmargin=*,widest=iii]
\item $\udf[\XB,\XX](m) \approx \Lambda_m$ for $m\in\NN$,
\item $\unc_m[\XB,\XX]\approx \aunc_m[\XB,\XX] \approx (\log m)^{1/q-1/r}$ for $m\ge 2$,
\item $\XB$ is $d_{1,q}(\ww)$-Hilbertian and $d_{1,r}(\ww)$-Besselian.
\item  $\XB$ is not $d_{1,q_1}(\ww)$-Hilbertian for any $q_1>q$ nor $d_{1,r_1}(\ww)$-Besselian for any $r_1<r$.
\end{enumerate}
\end{lemma}

\begin{proof}
The unconditional basis $\VB=(\vv_n)_{n=1}^\infty$ of $P_\sigma(\Sym)$ is $q$-Hilbertian and $r$-Besselian. Therefore, by Lemma~\ref{lem:gathered}, Lemma~\ref{lem:EstimatesRP}, Proposition~\ref{prop:DKKHilbertian} and Proposition~\ref{prop:DKKBesselian},
\[
d_{1,q}(\ww) \subseteq \YY:=\YY[\BB\oplus\VB, \Sym,\sigma] \subseteq d_{1,r}(\ww),
\]
provided that the ordered partition $\sigma=(\sigma_n)_{n=1}^\infty$ satisfies \eqref{eq:OPConditionA}. Choose $\sigma$ so that
\[
\log\left(\sum_{n=1}^m \abs{\sigma_n}\right) \lesssim m, \quad m\in\NN,
\]
also holds. Then, by \eqref{eq:AUNCCDS} and \cite{AADK2018}*{Proposition 3.8 and Theorem 3.17}, the unit vector system is an almost greedy basis of $\YY$ with fundamental function equivalent to $(\Lambda_m)_{m=1}^\infty$, and we have
\[
\aunc_m[\EB,\YY] \gtrsim (\log m)^{1/q-1/r}, \quad m\ge 2.
\]
Applying \cite{AADK2018}*{Theorem 3.6}, and using that $\Sym$ is isomorphic to its square, gives
\[
\YY\simeq Q_\sigma(\Sym) \oplus \XX \oplus P_\sigma(\Sym) \simeq \XX\oplus\Sym\simeq\XX.
\]
Combining Equation~\eqref{eq:domWLorentz}, Lemma~\ref{lem:equivalentnorm} and Corollary~\ref{cor:OptConditions} puts an end to the proof.
\end{proof}

Theorem~\ref{thm:AGSBlp}, of which Theorem~\ref{thm:CPAGlp}  is a simple consequence, lies within the line of research initiated by  Konyagin and Telmyakov \cite{KoTe1999} of finding conditional quasi-greedy bases in general Banach spaces. This topic has evolved towards the more specific quest of finding quasi-greedy bases with suitable conditionality parameters. The reader will find a detailed account of this process in the papers \cites{Woj2000,DKW2002,DKK2003,Gogyan2010, GHO2013, AAW2019,AADK2018,AAW2021b}.

\begin{theorem}\label{thm:AGSBlp}
Let $\XX$ be a Banach space with a Schauder basis $\BB$, and let $1<p<\infty$ be such that $\XX$ has a complemented subspace isomorphic to $\ell_p$. Let $1<q_0\le\min\{2,p\}$ and $\max\{2,p\}\le r_0<\infty$. Suppose  that $\BB$ is $q_0$-Hilbertian and  $r_0$-Besselian, and that $\unc_m[\XB,\XX] \lesssim m^{1/q_0-1/r_0}$ for $m\in\NN$. Then, $\XX$ has, for any $1<q\le q_0$ and $r_0\le r<\infty$,   an almost greedy basis $\XB$ such that
\begin{enumerate}[label=(\roman*),leftmargin=*,widest=iii]
\item $\udf[\XB,\XX](m) \approx m^{1/p}$ for $m\in\NN$,
\item $\unc_m[\XB,\XX] \approx (\log m)^{1/q-1/r}$ for $m\ge 2$,
\item $\XB$ is $\ell_{p,q}(\ww)$-Hilbertian and $\ell_{p,r}(\ww)$-Besselian, and
\item  $\XB$ is not $\ell_{p,q_1}$-Hilbertian for any $q_1>q$ nor $\ell_{p,r_1}$-Besselian for any $r_1<r$.
\end{enumerate}
\end{theorem}

\begin{proof}
Just combine Corollary~\ref{cor:CPlpBis} with Lemma~\ref{lem:BB}.
\end{proof}

\begin{proof}[Proof of Theorem~\ref{thm:CPAGlp}]
We apply Theorem~\ref{thm:AGSBlp} in the case where $\XX$ is $\ell_p$ and $\BB$ is its unit vector system, so that we can choose $q_0=\min\{2,p\}$ and $r_0=\max\{2,p\}$. If $\alpha_0=1/q_0-1/r_0$,  then for any  $\alpha\in[\alpha_0,1)$
there are $1<q\le q_0$ and $r_0\le r <1$ such that $1/q-1/r=\alpha$. Since $\alpha_0=\abs{1/2-1/p}$, where are done.
\end{proof}


\begin{bibdiv}
\begin{biblist}

\bib{AABBL2021}{article}{
      author={Albiac, Fernando},
      author={Ansorena, Jos\'{e}~L.},
      author={Berasategui, Miguel},
      author={Bern\'{a}, Pablo~M.},
      author={Lassalle, Silvia},
       title={Bidemocratic bases and their connections with other greedy-type
  bases},
        date={2021},
     journal={arXiv e-prints},
      eprint={2105.15177},
}

\bib{AABW2021}{article}{
      author={Albiac, Fernando},
      author={Ansorena, Jos\'{e}~L.},
      author={Bern\'{a}, Pablo~M.},
      author={Wojtaszczyk, Przemys{\l}aw},
       title={Greedy approximation for biorthogonal systems in quasi-{B}anach
  spaces},
        date={2021},
     journal={Dissertationes Math. (Rozprawy Mat.)},
      volume={560},
       pages={1\ndash 88},
}

\bib{AADK2018}{article}{
      author={Albiac, Fernando},
      author={Ansorena, Jos\'e~L.},
      author={Dilworth, Stephen~J.},
      author={Kutzarova, Denka},
       title={Building highly conditional almost greedy and quasi-greedy bases
  in {B}anach spaces},
        date={2019},
        ISSN={0022-1236},
     journal={J. Funct. Anal.},
      volume={276},
      number={6},
       pages={1893\ndash 1924},
         url={https://doi-org/10.1016/j.jfa.2018.08.015},
      review={\MR{3912795}},
}

\bib{AAW2019}{article}{
      author={Albiac, Fernando},
      author={Ansorena, Jos\'{e}~L.},
      author={Wojtaszczyk, Przemys{\l}aw},
       title={Conditional quasi-greedy bases in non-superreflexive {B}anach
  spaces},
        date={2019},
        ISSN={0176-4276},
     journal={Constr. Approx.},
      volume={49},
      number={1},
       pages={103\ndash 122},
         url={https://doi-org/10.1007/s00365-017-9399-x},
      review={\MR{3895765}},
}

\bib{AAW2021b}{article}{
      author={Albiac, Fernando},
      author={Ansorena, Jos\'{e}~L.},
      author={Wojtaszczyk, Przemys{\l}aw},
       title={On certain subspaces of {$\ell_p$} for {$0<p\leq1$} and their
  applications to conditional quasi-greedy bases in {$p$}-{B}anach spaces},
        date={2021},
        ISSN={0025-5831},
     journal={Math. Ann.},
      volume={379},
      number={1-2},
       pages={465\ndash 502},
         url={https://doi-org/10.1007/s00208-020-02069-3},
      review={\MR{4211094}},
}

\bib{AlbiacKalton2016}{book}{
      author={Albiac, Fernando},
      author={Kalton, Nigel~J.},
       title={Topics in {B}anach space theory},
     edition={Second Edition},
      series={Graduate Texts in Mathematics},
   publisher={Springer, [Cham]},
        date={2016},
      volume={233},
        ISBN={978-3-319-31555-3; 978-3-319-31557-7},
         url={https://doi.org/10.1007/978-3-319-31557-7},
        note={With a foreword by Gilles Godefroy},
      review={\MR{3526021}},
}

\bib{Altman1949}{article}{
      author={Al{\cprime}tman, M.~\v{S}.},
       title={On bases in {H}ilbert space},
        date={1949},
     journal={Doklady Akad. Nauk SSSR (N.S.)},
      volume={69},
       pages={483\ndash 485},
      review={\MR{0033976}},
}

\bib{ABW2021}{article}{
      author={Ansorena, Jos\'{e}~L.},
      author={Bello, Glenier},
      author={Wojtaszczyk, Przemys{\l}aw},
       title={Lorentz spaces and embeddings induced by almost greedy bases in
  superreflexive {B}anach spaces},
        date={2021},
     journal={arXiv e-prints},
      eprint={2105.09203},
        note={Accepted for publication in Israel Journal of Mathematics},
}

\bib{Babenko1948}{article}{
      author={Babenko, Konstantin~I.},
       title={On conjugate functions},
        date={1948},
     journal={Doklady Akad. Nauk SSSR (N. S.)},
      volume={62},
       pages={157\ndash 160},
      review={\MR{0027093}},
}

\bib{BerLof1976}{book}{
      author={Bergh, J{\"o}ran},
      author={L{\"o}fstr{\"o}m, J{\"o}rgen},
       title={Interpolation spaces. {A}n introduction},
   publisher={Springer-Verlag, Berlin-New York},
        date={1976},
        note={Grundlehren der Mathematischen Wissenschaften, No. 223},
      review={\MR{0482275}},
}

\bib{DKK2003}{article}{
      author={Dilworth, Stephen~J.},
      author={Kalton, Nigel~J.},
      author={Kutzarova, Denka},
       title={On the existence of almost greedy bases in {B}anach spaces},
        date={2003},
        ISSN={0039-3223},
     journal={Studia Math.},
      volume={159},
      number={1},
       pages={67\ndash 101},
         url={https://doi.org/10.4064/sm159-1-4},
        note={Dedicated to Professor Aleksander Pe{\l}czy\'nski on the occasion
  of his 70th birthday},
      review={\MR{2030904}},
}

\bib{DKKT2003}{article}{
      author={Dilworth, Stephen~J.},
      author={Kalton, Nigel~J.},
      author={Kutzarova, Denka},
      author={Temlyakov, Vladimir~N.},
       title={The thresholding greedy algorithm, greedy bases, and duality},
        date={2003},
        ISSN={0176-4276},
     journal={Constr. Approx.},
      volume={19},
      number={4},
       pages={575\ndash 597},
         url={https://doi-org/10.1007/s00365-002-0525-y},
      review={\MR{1998906}},
}

\bib{DKW2002}{article}{
      author={Dilworth, Stephen~J.},
      author={Kutzarova, Denka},
      author={Wojtaszczyk, Przemys{\l}aw},
       title={On approximate {$l_1$} systems in {B}anach spaces},
        date={2002},
        ISSN={0021-9045},
     journal={J. Approx. Theory},
      volume={114},
      number={2},
       pages={214\ndash 241},
         url={https://doi.org/10.1006/jath.2001.3641},
      review={\MR{1883407}},
}

\bib{Duo2001}{book}{
      author={Duoandikoetxea, Javier},
       title={Fourier analysis},
      series={Graduate Studies in Mathematics},
   publisher={American Mathematical Society, Providence, RI},
        date={2001},
      volume={29},
        ISBN={0-8218-2172-5},
         url={https://doi-org/10.1090/gsm/029},
        note={Translated and revised from the 1995 Spanish original by David
  Cruz-Uribe},
      review={\MR{1800316}},
}

\bib{GHO2013}{article}{
      author={Garrig\'os, Gustavo},
      author={Hern\'{a}ndez, Eugenio},
      author={Oikhberg, Timur},
       title={{L}ebesgue-type inequalities for quasi-greedy bases},
        date={2013},
        ISSN={0176-4276},
     journal={Constr. Approx.},
      volume={38},
      number={3},
       pages={447\ndash 470},
         url={https://doi-org/10.1007/s00365-013-9209-z},
      review={\MR{3122278}},
}

\bib{GW2014}{article}{
      author={Garrig\'os, Gustavo},
      author={Wojtaszczyk, Przemys{\l}aw},
       title={Conditional quasi-greedy bases in {H}ilbert and {B}anach spaces},
        date={2014},
     journal={Indiana Univ. Math. J.},
      volume={63},
      number={4},
       pages={1017\ndash 1036},
}

\bib{Gelbaum1951}{article}{
      author={Gelbaum, Bernard},
       title={A nonabsolute basis for {H}ilbert space},
        date={1951},
        ISSN={0002-9939},
     journal={Proc. Amer. Math. Soc.},
      volume={2},
       pages={720\ndash 721},
         url={https://doi.org/10.2307/2032069},
      review={\MR{43383}},
}

\bib{Gogyan2010}{article}{
      author={Gogyan, Smbat},
       title={An example of an almost greedy basis in {$L^1(0,1)$}},
        date={2010},
        ISSN={0002-9939},
     journal={Proc. Amer. Math. Soc.},
      volume={138},
      number={4},
       pages={1425\ndash 1432},
         url={https://doi-org/10.1090/S0002-9939-09-10169-7},
      review={\MR{2578535}},
}

\bib{GurGur1971}{article}{
      author={Gurari\u{\i}, Vladimir~I.},
      author={Gurari\u{\i}, N.~I.},
       title={Bases in uniformly convex and uniformly smooth {B}anach spaces},
        date={1971},
        ISSN={0373-2436},
     journal={Izv. Akad. Nauk SSSR Ser. Mat.},
      volume={35},
       pages={210\ndash 215},
        note={English translation in \emph{Bases in uniformly convex and
  uniformly flattened {B}anach spaces} Math. USSR Izv. \textbf{220} (1971) no.
  5},
      review={\MR{0283549}},
}

\bib{HMW1973}{article}{
      author={Hunt, Richard},
      author={Muckenhoupt, Benjamin},
      author={Wheeden, Richard},
       title={Weighted norm inequalities for the conjugate function and
  {H}ilbert transform},
        date={1973},
        ISSN={0002-9947},
     journal={Trans. Amer. Math. Soc.},
      volume={176},
       pages={227\ndash 251},
         url={https://doi.org/10.2307/1996205},
      review={\MR{312139}},
}

\bib{James1972}{article}{
      author={James, Robert~C.},
       title={Super-reflexive spaces with bases},
        date={1972},
        ISSN={0030-8730},
     journal={Pacific J. Math.},
      volume={41},
       pages={409\ndash 419},
         url={http://projecteuclid.org/euclid.pjm/1102968287},
      review={\MR{308752}},
}

\bib{KoTe1999}{article}{
      author={Konyagin, Sergei~V.},
      author={Temlyakov, Vladimir~N.},
       title={A remark on greedy approximation in {B}anach spaces},
        date={1999},
        ISSN={1310-6236},
     journal={East J. Approx.},
      volume={5},
      number={3},
       pages={365\ndash 379},
      review={\MR{1716087}},
}

\bib{KotheToeplitz1934}{article}{
      author={K{\"o}the, Gottfried},
      author={Toeplitz, Otto},
       title={Lineare {R}{\"a}ume mit unendlich vielen {K}oordinaten und
  {R}inge unendlicher {M}atrizen},
        date={1934},
        ISSN={0075-4102},
     journal={J. Reine Angew. Math.},
      volume={171},
       pages={193\ndash 226},
         url={https://doi-org/10.1515/crll.1934.171.193},
      review={\MR{1581429}},
}

\bib{LinTza1977}{book}{
      author={Lindenstrauss, Joram},
      author={Tzafriri, Lior},
       title={Classical {B}anach spaces. {I} -- sequence spaces},
      series={Ergebnisse der Mathematik und ihrer Grenzgebiete [Results in
  Mathematics and Related Areas]},
   publisher={Springer-Verlag, Berlin-New York},
        date={1977},
        ISBN={3-540-08072-4},
      review={\MR{0500056}},
}

\bib{LinTza1979}{book}{
      author={Lindenstrauss, Joram},
      author={Tzafriri, Lior},
       title={Classical {B}anach spaces. {II} -- function spaces},
      series={Ergebnisse der Mathematik und ihrer Grenzgebiete [Results in
  Mathematics and Related Areas]},
   publisher={Springer-Verlag, Berlin-New York},
        date={1979},
      volume={97},
        ISBN={3-540-08888-1},
      review={\MR{540367}},
}

\bib{Pel1960}{article}{
      author={Pe{\l}czy\'{n}ski, Aleksander},
       title={Projections in certain {B}anach spaces},
        date={1960},
        ISSN={0039-3223},
     journal={Studia Math.},
      volume={19},
       pages={209\ndash 228},
         url={https://doi-org/10.4064/sm-19-2-209-228},
      review={\MR{126145}},
}

\bib{PelSin1964}{article}{
      author={Pe{\l}czy\'{n}ski, Aleksander},
      author={Singer, Ivan},
       title={On non-equivalent bases and conditional bases in {B}anach
  spaces},
        date={1964/65},
        ISSN={0039-3223},
     journal={Studia Math.},
      volume={25},
       pages={5\ndash 25},
         url={https://doi-org/10.4064/sm-25-1-5-25},
      review={\MR{179583}},
}

\bib{Woj2000}{article}{
      author={Wojtaszczyk, Przemys{\l}aw},
       title={Greedy algorithm for general biorthogonal systems},
        date={2000},
        ISSN={0021-9045},
     journal={J. Approx. Theory},
      volume={107},
      number={2},
       pages={293\ndash 314},
         url={https://doi-org/10.1006/jath.2000.3512},
      review={\MR{1806955}},
}

\end{biblist}
\end{bibdiv}


\end{document}